\documentclass[a4paper]{article}

\usepackage[english]{babel}
\usepackage[utf8x]{inputenc}
\usepackage[T1]{fontenc}

\usepackage[a4paper,top=3cm,bottom=2cm,left=3cm,right=3cm,marginparwidth=1.75cm]{geometry}

\usepackage{amsmath}
\usepackage{amssymb}
\usepackage{graphicx}
\usepackage[colorinlistoftodos]{todonotes}
\usepackage[colorlinks=true, allcolors=black]{hyperref}
\usepackage{amsthm}
\usepackage{amsfonts}
\usepackage{MnSymbol}
\usepackage{enumitem}
\usepackage{mathtools}
\usepackage{placeins}
\usepackage{extarrows}
\usepackage{sidecap}
\usepackage[font=small,labelfont=bf]{caption}
\usepackage{subfig}
\usepackage{mathrsfs}
\usepackage{dutchcal}

\theoremstyle{definition}
\newtheorem{definition}{Definition}[section]
\newtheorem*{implementation}{Implementation}
\newtheorem*{notation}{Notation}
\newtheorem{theorem}{Theorem}[section]
\newtheorem{lemma}{Lemma}[section]  
\newtheorem{proposition}{Proposition}[section]
\newtheorem{example}{Example}[section]
\newtheorem{corollary}{Corollary}[section]
\theoremstyle{remark}
\newtheorem*{remark}{Remark}
\newtheorem*{remarks}{Remarks}

\newenvironment{myproof1}[1][\proofname]{%
  \proof[\emph{\textbf{Proof of Theorem 1.1} (Arithmetic-Geometric Sums)}]%
}{\endproof}
\newenvironment{myproof2}[1][\proofname]{%
  \proof[\emph{\textbf{Proof of Theorem 1.2} (Shifted Powers of $2$)}]%
}{\endproof}
\newenvironment{myproof3}[1][\proofname]{%
  \proof[\emph{\textbf{Proof of Theorem 1.3} (Extended Triangular Numbers)}]%
}{\endproof}
\newenvironment{myproof4}[1][\proofname]{%
  \proof[\emph{\textbf{Proof of Theorem 1.4} (Divisor Functions)}]%
}{\endproof}
\newenvironment{myproof5}[1][\proofname]{%
  \proof[\emph{\textbf{Proof of Theorem 1.5} (Almost Divisible Numbers)}]%
}{\endproof}

\DeclareMathOperator{\lcm}{lcm}
\newcommand{\q}{\text{ }}
\newcommand{\N}{\mathbb{N}}
\newcommand{\Ap}{\text{Ap}}
\newcommand{\Gen}{\text{Gen}}
\newcommand{\Bal}{\text{Bal}}
\newcommand{\Rem}{\text{Rem}}
\newcommand{\bx}{\boldsymbol{x}}

\newcommand{\ba}{\boldsymbol{a}}
\newcommand{\bu}{\boldsymbol{u}}
\newcommand{\bigO}{\mathcal{O}}
\newcommand{\lf}{\left\lfloor}
\newcommand{\rf}{\right\rfloor}
\newcommand{\lc}{\left\lceil}
\newcommand{\rc}{\right\rceil}

\title{A Graphical Approach to Finding the Frobenius Number, Genus and Hilbert Series of a Numerical Semigroup}

\author{Alexandru Pascadi}

\begin{document}
\maketitle

\begin{abstract}
This paper proposes a new, visual method to study numerical semigroups and the Frobenius problem. The method is based on building a so-called reduction graph, whose nodes usually correspond to monogenic semigroups, and whose edges can have multiple inputs and outputs. If such a construction is possible, then determining whether the studied semigroup is symmetric, or finding explicit forms of its Ap\'ery set and Hilbert series, is reduced to straightforward computations assisted by a MAPLE program we made available on arXiv. This approach applies to many of the cases considered in literature, including semigroups generated by arithmetic and geometric sequences, compound sequences, progressions of the form $a^n, a^n + a, \ldots, a^n + a^{n-1}$, triangular and tetrahedral numbers, certain Fibonacci triplets, etc.

After explaining the general approach in more detail, the paper studies the types of edges that can be used as building blocks of a reduction graph, as well as a series of operations that serve to modify or combine valid reduction graphs. In the end of the paper, we use these techniques to solve the Frobenius problem for $7$ new classes of numerical semigroups.
\end{abstract}

\section{Numerical Semigroups and the Frobenius Problem}

The main goal of this paper is to characterize a wide variety of numerical semigroups \cite{Numerical}, using a graphical representation for their structures. With this approach, finding the Frobenius number, genus, Ap\'ery sets, Hilbert series and other attributes of a numerical semigroup (all to be reviewed shortly) will be reduced to constructing a valid (hyper)graph with certain properties. The nodes of these graphs will be sets of nonnegative integers, while the edges will represent equalities between sums of sets, so let us start by presenting some related common notations:

\begin{notation}[Sums and Direct Sums of Sets]
If $A$ and $B$ are sets (for our purposes, of nonnegative integers) equipped with an additive operation, we denote by $A + B$ the following set:
\begin{equation}
\begin{split}
A + B &:= \{a + b : a \in A, b \in B\}
\end{split}
\end{equation}

Conversely, given a set $C$ equipped with the same additive operation, such that for all $c \in A + B$ there exist unique $a \in A$ and $b \in B$ with $c = a + b$, we write:
\begin{equation}
C = A + B =: A \oplus B
\end{equation}

In this case, we say that $C$ is an (internal) direct sum of $A$ and $B$, borrowing terminology from abstract algebra. We note that usual sum and direct sum do not generally associate, in the sense that we can only rewrite $A + (B \oplus C)$ as $(A + B) \oplus C$ provided that the latter direct sum is defined.
\end{notation}

\begin{remark}
Infinite sums of sets can also be defined when all the sets involved contain $0$, by considering all the possible finite sums: $\sum_{i \in I} A_i := \{a_1 + \ldots + a_n : a_k \in A_{i_k}, \{i_1, \ldots, i_n\} \subseteq I, n \geq 1\}$.
\end{remark}

\begin{definition}[Numerical Semigroups]
A semigroup $S$ of nonnegative integers under addition is called \emph{numerical} iff $\N \setminus S$ is finite (where $\N$ is the set of all nonnegative integers). An equivalent condition \cite{Numerical} is that $\gcd(S) = 1$, or that there exists $A \subseteq \N \setminus \{0\}$ with $\gcd(A) = 1$ that \emph{generates} $S$:
\begin{equation}
S = \langle A \rangle := \sum_{a \in A} \langle a \rangle = \sum_{a \in A} \{ an : n \in \N \}
\end{equation}

The semigroups $\langle a \rangle$ generated by one element are called \emph{monogenic}. It is known \cite{Numerical} that any numerical semigroup is finitely generated, and in fact has a unique and finite \emph{minimal system of generators} (i.e., a set $A$ that generates $S$ such that no proper subset of $A$ generates $S$). We are interested in studying the following characteristics of numerical semigroups:

\begin{enumerate}[leftmargin = 5mm]
\item The \emph{Ap\'ery set} \cite{Apery} of a numerical semigroup $S$ in terms of some $a \in S \setminus \{0\}$ is defined by:
\begin{equation}
\Ap(S, a) := \{s \in S : s - a \not\in S \}
\end{equation}

An essential property (and equivalent definition) of the Ap\'ery set is that:
\begin{equation}
S = \langle a \rangle \oplus \Ap(S, a)
\end{equation}

Indeed, the '$\supseteq$' inclusion holds because $\Ap(S, a) \subseteq S$ and $a \in S$, while the '$\subseteq$' inclusion follows directly from $(4)$. This sum is a direct sum because $\Ap(S, a)$ contains exactly one integer from each residue class modulo $a$ (in light of $(4)$, since $S$ contains all sufficiently large integers). In fact, $\Ap(S, a)$ contains precisely the elements: $\min\left(\{an + r : n \geq 0\} \cap S\right)$, for $0 \leq r < a$.

\item The \emph{set of gaps} in $S$ is simply the complement of $S$ in the nonnegative integers:
\begin{equation}
G(S) := \N \setminus S = \{x - an : x \in \Ap(S, a), n \geq 1\} \cap \N
\end{equation}

The second equality above holds for any $a \in S \setminus \{0\}$, because any nonnegative integer that is not in $S$ must be strictly smaller than the element of $\Ap(S, a)$ with the same residue modulo $a$ (otherwise, it could be written as $x + an$ for some $x \in \Ap(S, a)$, so it would belong to $S$ by $(5)$). Conversely, any element from the RHS of $(6)$ cannot belong to $S$, as that would contradict $(5)$.

\item The \emph{Frobenius number} of $S$ is the maximal integer that does not belong to $S$, denoted by $F(S)$. If $\gcd(a_1, \ldots, a_n) = 1$, we write $F(a_1, \ldots, a_n) := F(\langle a_1, \ldots, a_n \rangle)$. By $(6)$, when $S \neq \N$, one has:
\begin{equation}
F(S) = \max \left(G(S)\right) = \max \left(\Ap(S, a)\right) - a
\end{equation}

When $S = \N$, one has $F(S) = -1 = \max(\{0, \ldots, a-1\}) - a$, so the same formula applies. We note that $S$ is called $\emph{irreducible}$ if it is not an intersection of larger numerical semigroups, $\emph{symmetric}$ if it is irreducible and $F(S)$ is odd, and \emph{pseudo-symmetric} if it is irreducible and $F(S)$ is even.

\item The \emph{genus} of $S$ is the cardinality of its complement (which is finite), denoted:
\begin{equation}
g(S) := |G(S)|
\end{equation}

If $\gcd(a_1, \ldots, a_n) = 1$, we write $g(a_1, \ldots, a_n) := g(\langle a_1, \ldots, a_n \rangle)$. It is known \cite{Numerical} that $S$ is symmetric iff $g(S) = \frac{F(S)+1}{2}$, and pseudo-symmetric iff $g(S) = \frac{F(S)+2}{2}$.

\begin{remark}
Some authors \cite{Compound, Tripathi} use the notation $g(S)$ for the Frobenius number, and $n(S)$ for the genus of a numerical semigroup, while others \cite{Mersenne, Four, Triangular} prefer the above choice of notation.
\end{remark}

\item Given $k \geq 0$, the $k^{th}$ \emph{gaps' power sum} (non-standard terminology) of $S$ is:
\begin{equation}
S_k(S) := \sum_{n \in G(S)} n^k
\end{equation}

\item The \emph{Hilbert series} of a numerical semigroup $S$ is the (ordinary) generating function of $S$ (i.e., of the characteristic function corresponding to $S$), written as:
\begin{equation}
H_S(X) := \Gen(S; X) = \sum_{s \in S} X^s
\end{equation}

In particular, the generating polynomial of the set of gaps in $S$ can be found as:
\begin{equation}
\begin{split}
\Gen(G(S); X) &= \left(1 + X + X^2 + X^3 + \ldots\right) - H_S(X) \\
&= \frac{1}{1-X} - \frac{1}{1-X^a} \Gen(\Ap(S, a); X)
\end{split}
\end{equation}

The last equality is due to relation $(5)$ and the following simple lemma:
\end{enumerate}
\end{definition}

\begin{lemma}
If $A$, $B$ and $C$ are sets of nonnegative integers such that $A \oplus B = C$, then:
\begin{equation}
\Gen(A; X) \cdot \Gen(B; X) = \Gen(C; X)
\end{equation}
\end{lemma}

\begin{proof}
Since every element of $C$ can be uniquely written as $a + b$ with $a \in A$ and $b \in B$, we get that:
\begin{equation}
\sum_{a \in A} X^{a}\sum_{b \in B} X^{b} = \sum_{\substack{a \in A \\ b \in B}} X^{a+b} = \sum_{c \in C} X^c
\end{equation}
\end{proof}

\begin{remark}
Among the attributes of numerical semigroups previously listed, the most popular one is probably the Frobenius number, the computation of which is known as the \emph{Frobenius Coin Problem} \cite{Frobenius}. The "coin" terminology comes from the following interpretation of the problem: given a sequence $a_1, a_2, \ldots, a_n$ of coin denominations with greatest common divisor $1$, we wish to find the greatest (integer) amount of change that cannot be created from these denominations (assuming we have infinitely many coins of each denomination available). Of course, the coin denominations are precisely the generators of the numerical semigroup to be studied.
\end{remark}

This problem was originally solved by Sylvester \cite{Sylvester, Sylvester2} when $n = 2$. There is considerable progress \cite{Three} on the case $n = 3$, as well as many asymptotical results \cite{Asym, Asym2}, but the general question is known to be NP-hard \cite{NP}. In consequence, most of the literature on the Coin problem focuses on finding the Frobenius numbers of semigroups generated by special sequences of integers: arithmetic \cite{Arithm}, modified arithmetic \cite{Modified}, geometric \cite{Geom} and compound \cite{Compound} sequences; Fibonacci \cite{Fibo}, Mersenne \cite{Mersenne}, repunit \cite{Repunit}, Thabit \cite{Thabit, Four}, Cunningham \cite{Four}, triangular and tetrahedral \cite{Triangular} numbers; triplets with certain divisibility constraints \cite{Pakornrat}, some shifted powers \cite{Tripathi}, etc. The methods of the current paper are applicable to most of the aforementioned sequences, as well as to some new ones:

\begin{theorem}[Arithmetic-Geometric Sums]
Let $a, b, d, n$ be positive integers such that $d \divides b$ and $\gcd(a, b) = 1$. First, define $S := \left\langle a^n, a^n + a^{n-1}d, a^n + 2a^{n-1}d, \ldots, a^n + a^{n-1}b \right\rangle$, a semigroup whose generators lie in an arithmetic sequence. Then, if $b \leq a$, continue this sequence of generators by defining $S_1 := \left\langle S \cup \{ a^n + a^{n-2}b^2, \ldots, a^n + ab^{n-1},  a^n + b^n\} \right\rangle$. Similarly, making abstraction of the condition $b \leq a$, let $S_2 := \left\langle S \cup \{ a^n + a^{n-1}b + a^{n-2}b^2, \ldots, a^n + a^{n-1}b + \ldots + b^n \} \right\rangle$. We claim that:
\begin{align*}
F(S_1) &= a^{n-1}\left(a\lc \frac{(a-1)d}{b} \rc + ad - a - d\right) + (a-1)\frac{(n-1)a^n(a-b) + b^2\left(a^{n-1}-b^{n-1}\right)}{a-b}
\\
F(S_2) &= a^{n-1}\left(a\lc \frac{(a-1)d}{b} \rc + ad - a - d\right) + (a-1)\frac{(n-1)a^{n+1}(a-b) - b^3\left(a^{n-1} - b^{n-1}\right)}{(a-b)^2}
\end{align*}

In both cases, setting $n = 1$ yields the case of arithmetic sequences \cite{Arithm}. When $d = b$, the generators of $S_1$ lie in a geometric sequence shifted by $a^n$, while the generators of $S_2$ are precisely the partial sums of a geometric sequence: $\left\langle a^n, a^n + a^{n-1}b, \ldots, a^n + a^{n-1}b + \ldots + b^n\right\rangle$. When $d = b = 1$, $S_1$ coincides with a semigroup initially studied by A. Tripathi \cite{Tripathi}: $\left\langle a^n, a^n + a, \ldots, a^n + a^{n-1} \right\rangle$.
\end{theorem}

\begin{theorem}[Shifted Powers of $2$]
Let $n$ be a positive integer and $0 \leq k \leq \nu_2(n)+1$, where $\nu_2(n)$ is the maximal exponent of a power of $2$ dividing $n$. Then $S := \left\langle n, n+2^0, n+2^1, \ldots, n+2^k \right\rangle$ is a numerical semigroup with Frobenius number:
\[
F(S) = \begin{cases}
\frac{n^2}{2^k} + (k - 1)n - 1, \quad &\text{if $k \leq \nu_2(n)$;} \\[3pt]
\frac{n^2}{2^k} + \left(k - \frac{3}{2}\right)n - 1, \quad &\text{if $k = \nu_2(n)+1$.} \\
\end{cases}
\]
\end{theorem}

\begin{theorem}[Extended Triangular Numbers]
Let $S := \left\langle \left\{ \frac{(n+i)(n+(i\% 2)+1)}{2} : 0 \leq i \leq k \right\}\right\rangle$ for some $n \geq 1$ and $k \geq 3$, where $i\% 2 \in \{0, 1\}$ denotes the residue of $i$ modulo $2$ (borrowing notation from computer science). Then, $S$ is a numerical semigroup with the following Frobenius number:
\[
F(S) = \begin{cases}
\lc \frac{n-2}{2\lf k/2 \rf} \rc T_n + \lc \frac{n}{\lf (k-1)/2 \rf} \rc T_{n+1} + n^2 + n - 1, \quad &\text{if $n$ is even;} \\[7pt]
\lc \frac{n-1}{\lf k/2 \rf} \rc T_n + \lc \frac{n-1}{2\lf (k-1)/2 \rf} \rc T_{n+1} + n^2 - 2, \quad &\text{if $n$ is odd.} \\
\end{cases}
\]

Above, we denoted $T_n := \frac{n(n+1)}{2}$. We note that for $k = 3$, the semigroup $S = \langle T_n, T_{n+1}, T_{n+1}, T_{n+2} \rangle$ $=\langle T_n, T_{n+1}, T_{n+2} \rangle$ is generated by $3$ consecutive triangular numbers, a case initially studied in \cite{Triangular}.
\end{theorem}

\begin{theorem}[Divisor Functions]
Let $n$, $t$ be positive integers such that for any distinct \emph{maximal} prime powers $p^k, q^l$ dividing $n$, one has $\gcd\left(\frac{p^{t(k+1)}-1}{p^t-1}, \frac{q^{t(l+1)}-1}{q^t-1} \right) = 1$. Considering $1$ a prime power, the semigroup $S := \left\langle \left\{ \sigma_t(m) : \frac{m}{n} \text{ is a prime power} \right\} \right\rangle$ is numerical and has the Frobenius number:
\[
F(S) = \sigma_t(n) \left(-1 + \sum_{1 < p^k \divides\divides n} \frac{p^{2t(k+1)} - p^t}{p^{t(k+1)} - 1}\right)
\]

Above, we used the notations: $\sigma_t(n) := \sum_{d \divides n} d^t$, and $p^k \divides\divides n \iff (p^k \divides n \text{, but } p^{k+1} \not \divides n)$.
\end{theorem}

\begin{theorem}[Almost Divisible Numbers]
Given positive integers $m$ and $n$, we say that $m$ is \emph{almost divisible} by $n$, written $m::n$, if there exists a prime power $p^k$ such that $n \divides p^k m$ (that is, $m$ is divisible by $n$ up to a prime power factor). Then for a fixed positive integer $n$, the semigroups $S_\leq := \langle\left\{ m \in \N : m \leq n, m::n \right\}\rangle = \langle\left\{ m \in \N : m::n \right\}\rangle$ and $S_\geq := \langle\left\{ m \in \N : m \geq n, m::n \right\}\rangle$ are numerical and have the following Frobenius numbers:
\[
F(S_\leq) = n\left(-1 + \sum_{1 < p^k \divides \divides n} \frac{p^k - 1}{p^k} \right) \qquad\qquad
F(S_\geq) = n\left(-1 + \sum_{1 < p^k \divides \divides n} \frac{2p^k - 1}{p^k}\right)
\]

\end{theorem}

Proofs for Theorems $1.1$ to $1.5$ are given in Section $6$, using the graphical approach developed in Sections $2$ to $5$. We note that the first two theorems can also be tackled using a recursive formula of Brauer and Shockley \cite{Frobenius}, used to compute the Frobenius number of $\langle a_1, \ldots, a_n \rangle$ when $\gcd(a_1, \ldots, a_{n-1}) > 1$. However, as shown in Corollary $5.1$, that formula is itself a simple consequence of our graphical method, together with similar relations about genera and Hilbert series.

Our approach also helps eliminate certain constraints on the semigroups already studied in literature: in particular, to allow negative common differences in modified arithmetic sequences \cite{Modified}, and to disregard the order restrictions on compound sequences \cite{Compound}. As a final note, one class of semigroups not covered by our methods is given by $\left\langle (n-1)^k, n^k, (n+1)^k \right\rangle$ \cite{Powers}, with $n \geq 2$, $k \geq 1$.

\section{Generalizing Ap\'ery Sets: Reductions and Remainder Sets}

In light of relations $(7)-(11)$, once we find a closed form expression for $\Gen(\Ap(S, a); X)$, we will find it very easy to characterize the numerical semigroup $S$. Therefore, the only hard part of solving the Frobenius problem (and associated questions) remains determining a simple expression of some Ap\'ery set $\Ap(S, a)$. We will do this by studying a type of relations that generalize equation $(5)$, called \emph{reductions}; these reductions turn out to have a compact visual representation as edges of a special graph, and obey convenient composition properties that can be manipulated graphically:

\begin{definition}[Reductions]
Let $A$ and $B$ be nonempty sets of nonnegative integers containing $0$. We say that $B$ can be \emph{reduced} by (or with respect to) $A$ if there exists a finite set $R$ such that:
\begin{equation}
A + B = A \oplus R
\end{equation}

We will refer to the equality above as a \emph{reduction}, to $R$ as the \emph{remainder set} of the reduction, and to the cardinality of the remainder set, $w := |R|$, as the \emph{weight} of the reduction. Note that since $A$ and $B$ contain $0$, the remainder set $R$ must also contain $0$, so $w \geq 1$.

\end{definition}

\begin{remark}
 In light of Lemma $1.1$, the remainder set corresponding to a reduction is \textbf{unique}, since one can take inverses of the power series $\Gen(A; X)$.
\end{remark}

\begin{example}
Given coprime positive integers $a$ and $b$, one has the so-called \emph{binary reduction}:
\begin{equation}
\langle a \rangle + \langle b \rangle = \langle a \rangle \oplus \{br : 0 \leq r < a\}
\end{equation}

The identity above follows because one can rewrite any $ax + by$ with $x, y \in \N$ as $a(x + bq) + br$, where $y = aq + r$ and $0 < r \leq a$. Therefore, we can restrict the value of $y$ to a residue modulo $a$ without losing elements of the set $\langle a \rangle + \langle b \rangle = \{ax + by : x, y \geq 0\}$. The sum in the RHS of $(15)$ is a direct sum since there are no repetitions modulo $a$ within $\{br : 0 \leq r < a\}$ (using that $\gcd(a, b) = 1$).

One can use this observation to solve the Coin problem for two coin denominations \cite{Sylvester, Sylvester2}, i.e. to find the Frobenius number of $\langle a, b \rangle = \langle a \rangle + \langle b \rangle$. By $(15)$, $\Ap(\langle a, b \rangle, a)$ is precisely the remainder set of our reduction: $\{br : 0 \leq r < a\}$, so by $(7)$, $F(\langle a, b \rangle)$ should be $a(b-1) - a = ab - a - b$.

In general, any Ap\'ery set is an instance of a remainder set, obtained when the set $A$ is a monogenic semigroup $\langle a \rangle$ and the weight of the reduction is $a$. Indeed, if $S$ is a numerical semigroup and $a \in S$, then we can write $S = \langle \{a\} \cup G \rangle$ for some remaining system of generators $G$. Then $S = \langle a \rangle + \langle G \rangle = \langle a \rangle \oplus \Ap(S, a)$, which fits the pattern of a reduction for $B = \langle G \rangle$. This pattern turns out to be present in many other ways within the structure of numerical semigroups, once we allow the set $\langle a \rangle$ to be replaced by a more general set $A$. This is what motivated the definition of reductions, along with the concise graphical representation further described.
\end{example}

\begin{notation}
The reductions described in relations $(14)$ and $(15)$ (note that the latter is a particular case of the former) can be represented as the following \emph{reduction edges}, where $w = |R|$:

\begin{figure}[h]
\centering
\includegraphics[scale = 0.85]{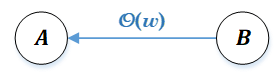}
\hspace{1cm}
\includegraphics[scale = 0.85]{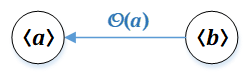}
\caption{Reduction Edge (left) vs. Binary Reduction Edge (right)}
\end{figure}
\FloatBarrier

Here we used a variation of the big-O notation: for our purposes, the big-curly-O will indicate the cardinality of a set whose exact form is known, but omitted for brevity. For instance, we say that the set $R$ is $\bigO(w)$ because it has cardinality $w$, and we may write the reduction from $(14)$ in a less comprehensive version as $A + B = A \oplus \bigO(w)$. Writing $\bigO(w)$ next to an edge indicates that the remainder set of that reduction is $\bigO(w)$, and thus that the \emph{weight of the edge} (i.e. the weight of the reduction) is $w$. Luckily, although the exact forms of remainder sets are essential in studying numerical semigroups, one can get very far by looking only at the weights of the reductions used.

On the other hand, it will be helpful to break down the sets $A$ and $B$ into finite or countable sums of sets of nonnegative integers containing $0$, say $A = A_1 + \ldots + A_k + \ldots$ and $B = B_1 + \ldots + B_q + \ldots$; for example, if $A$ is a semigroup, one can choose the $A_i$'s as the monogenic semigroups corresponding to the generators of $A$. This allows for a more detailed representation of reduction edges:

\begin{figure}[h]
\centering
\includegraphics[scale = 0.85]{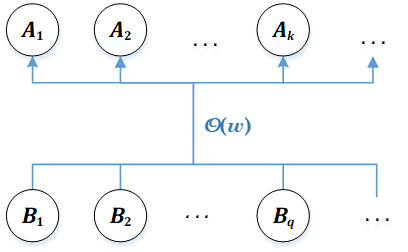}
\hspace{1cm}
\includegraphics[scale = 0.85]{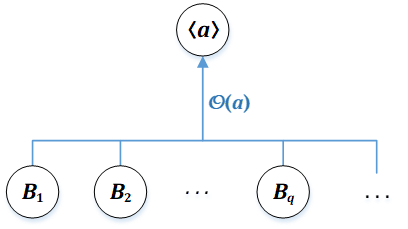}
\caption{General Reduction Edge (left) vs. Ap\'ery Reduction Edge (right)}
\end{figure}
\FloatBarrier
\end{notation}

Figure $2$ (left) displays a general reduction edge, which describes the same reduction $A + B = A \oplus \bigO(w)$ as Figure $1$ (left), but in more detail:
\[
\left(A_1 + \ldots + A_k + \ldots\right) + \left(B_1 + \ldots + B_q + \ldots\right) = \left(A_1 + \ldots + A_k + \ldots\right)  \oplus \bigO(w)
\]

We emphasize that Figure $2$ (left) is regarded as a \emph{single} edge corresponding to a \emph{single} reduction, although it may have multiple inputs and outputs. The advantage of splitting $A$ and $B$ into sums of other sets is that each $A_i$ (respectively $B_i$) can now take part in other reductions, independently of the sets $A_j$ (respectively $B_j$) with $j \neq i$. This brings more freedom in constructing graphs with complex networks of reduction edges.

As mentioned before, an important category of remainder sets are Ap\'ery sets, which lead to so-called \emph{Ap\'ery reductions} of the form: $\langle a \rangle + B = S = \langle a \rangle \oplus \Ap(S, a)$. Writing $B = B_1 + \ldots + B_q + \ldots$ once again, we obtain the graphical representation from Figure $2$ (right), which describes the equality:
\[
\langle a \rangle + \left(B_1 + \ldots + B_q + \ldots\right) = \langle a \rangle \oplus \bigO(a)
\]

Still, Figure $2$ (right) by itself does not display the full complexity behind finding the Ap\'ery set $\Ap(S, a)$, nor does it provide a simple expression for $\Gen(\Ap(S, a); X)$. Ideally, the reduction of $B$ with respect to $\langle a\rangle$ could be decomposed into a series of other simpler reductions, whose remainder sets are easy to describe. The resulting hypergraph, which captures the structure of $\Ap(S, a)$ much more thoroughly, will be called a \emph{reduction graph} of the semigroup $S$, to be formalized in the next section. For a preliminary visualization of this concept, Figure $3$ shows two reduction graphs of the semigroup $S = \langle 9, 12, 15, 20 \rangle$, based on the following reductions that will be explained later:
\begin{align*}
&\langle 9 \rangle + \left(\langle 12 \rangle + \langle 15 \rangle \right) = \langle 9 \rangle \oplus \bigO(3)
\qquad \qquad \qquad \q \hspace{0.5mm}
\langle 20 \rangle + \langle 15 \rangle = \langle 20 \rangle \oplus \bigO(4)
\\
&
\langle 12 \rangle + \langle 20 \rangle = \langle 12 \rangle \oplus \bigO(3)
\qquad \qquad \quad \quad
\langle 15 \rangle + \left(\langle 12 \rangle + \langle 9 \rangle \right) = \langle 15 \rangle \oplus \bigO(5)
\end{align*}

\begin{figure}[h]
\centering
\includegraphics[scale = 0.8]{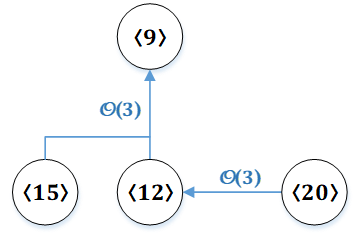}
\hspace{1cm}
\includegraphics[scale = 0.8]{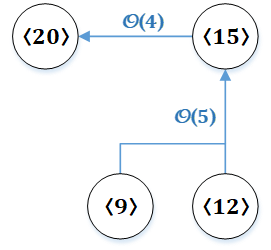}
\caption{Numeric Examples of Reduction Graphs}
\end{figure}
\FloatBarrier

\section{Reduction Graphs and Computational Aid}

\begin{definition}[Reduction Graphs]
Suppose that $G$ is an \emph{acyclic} weighted oriented hypergraph with the following properties:
\begin{enumerate}
\item The nodes/vertices of $G$, gathered in $V(G)$, are nonempty sets of nonnegative integers containing $0$. $V(G)$ may be infinite and may contain repetitions (as a multicollection of sets).

\item The edges of $G$, gathered in $E(G)$, are reduction edges with the form discussed in the previous section (see Figure $2$). Each edge $e \in E(G)$ can have any number of inputs and outputs, and carries a weight $w(e)$ equal to the cardinality of its corresponding remainder set, a set denoted by $\Rem(e)$. To emphasize that the weights are cardinalities of remainder sets, we shall write $\bigO(w)$ (rather than $w$) near the graphical representation of the edge. We also require that $E(G)$ is \textbf{finite}.

\item All nodes of the graph except for one have outdegree equal to $1$. The remaining node, called the \emph{root} of $G$, has outdegree $0$ and must be a \textbf{monogenic} semigroup. Since $G$ is acyclic and $E(G)$ is finite, this means that every path in $G$ eventually terminates at the root. We will refer to the positive integer that generates the monogenic semigroup in the root as the \emph{root generator} of $G$, denoted $r(G)$.
\end{enumerate}

Under these conditions, we say that $G$ is a \emph{reduction graph} of the set (or \emph{describing} the set) $S(G) := \sum_{X \in V(G)} X$. In practice, $S(G)$ will usually be a numerical semigroup.

\end{definition}

\begin{example}[Geometric Sequences]
Given coprime positive integers $a$ and $b$, let $V(G)$ consist of all nodes of the form $\langle a^{n-k} b^k \rangle$ for $0 \leq k \leq n$. Then $S(G) = \sum_{0 \leq k \leq n} \langle a^{n-k} b^k \rangle = \langle a^n, a^{n-1}b, \ldots, b^n \rangle$ is the numerical semigroup generated by a geometric progression \cite{Geom} (note that $\gcd(a^n, b^n) = 1$). To complete the definition of the reduction graph $G$, consider the following structure of edges:

\begin{figure}[h]
\centering
\includegraphics[scale = 0.85]{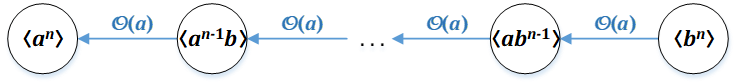}
\caption{Reduction Graph for a Geometric Sequence}
\end{figure}

More precisely, $E(G)$ consists of $n$ reduction edges of weight $a$. Each edge goes from $\langle a^{n-k-1}b^{k+1} \rangle$ to $\langle a^{n-k}b^{k} \rangle$ for some $0 \leq k < n$, describing a reduction equivalent to $(15)$ after a scaling by $a^{n-k-1}b^{k}$:
\[
\langle a^{n-k}b^k \rangle + \langle a^{n-k-1}b^{k+1} \rangle = \langle a^{n-k}b^k \rangle \oplus a^{n-k-1}b^{k} \{br : 0 \leq r < a \}
\]

The root generator of $G$ is $a^n$, which equals the product of all the weights of the edges in $E(G)$. As we shall see soon (in Theorem $3.1$ and Corollary $3.1$), this property of a reduction graph allows us to determine the Ap\'ery set, Frobenius number, genus, etc. of the studied numerical semigroup $S(G)$. A variation of this example arises for \emph{composed} geometric sequences, described in Figure $5$ (where we make the assumption that $\gcd(a, b) = \gcd(a, c) = \gcd(c, d) = 1$):

\begin{figure}[h]
\centering
\includegraphics[scale = 0.85]{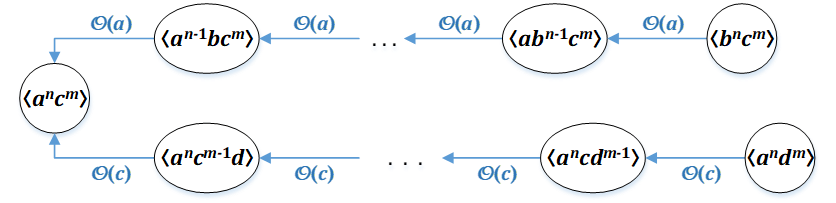}
\caption{Reduction Graph for Composed Geometric Sequences}
\end{figure}
\FloatBarrier

Similarly, one could \emph{compose} several different geometric sequences to generate a numerical semigroup (the reason behind calling these \emph{composed} sequences will be made clear in Section $5$), or mix them with arithmetic progressions, Mersenne numbers, and so on; there is a wide range of possibilities. The motivation behind this graphical formulation lies in the following proposition:
\end{example}

\begin{proposition}
If $G$ is a reduction graph, the following equality holds:
\begin{equation}
S(G) = \langle r(G) \rangle + \sum_{e \in E(G)} \Rem(e)
\end{equation}

We note that the RHS is a sum between a monogenic semigroup $\langle r(G) \rangle$ and a finite set. If this is a direct sum, then the latter finite set must coincide by $(5)$ with the Ap\'ery set $\Ap(S(G), r(G))$.
\end{proposition}

\begin{proof}
One should understand a reduction graph as a sequence of reductions to be applied iteratively, arranged in a partial chronological order indicated by the direction of the paths; it is essential that the graph is acyclic so that we don't get stuck in an infinite loop of reductions. More precisely, before any reduction is performed, the set $S(G)$ can be expressed by definition as:
\begin{equation}
S(G) = \sum_{X \in V(G)} X
\end{equation}

We will apply an algorithm that deconstructs the reduction graph $G$ while simplifying the sum in $(17)$, until we are left with the sum in $(16)$. Initially, denote $S := S(G)$ (we need this notation because $S(G)$ will change during the algorithm) and $R := \{ 0 \}$ (this will represent a cumulative remainder set; initially, the remainder is null). The equality $S = \left(\sum_{X \in V(G)} X\right) + R$ will be an invariant of our algorithm. Then, at each iteration, perform the following steps:

\begin{enumerate}
\item If possible, choose an edge $e \in E(G)$ such that the input nodes of $e$ have indegree $0$. Let $B = \sum_{j \in J} B_j$ be the sum of the input sets of $e$ and $A = \sum_{i \in I} A_i$ be the sum of the output nodes of $e$. Since $e \in E(G)$ and we assume $G$ is in a valid state (no dangling edges), all the inputs and outputs of $e$ are currently nodes in the graph $G$.

\item By Definition $2.1$, we have $A + B = A + \Rem(e)$ (we shall ignore the direct sum for now). Consequently, within the sum $S = \left( \sum_{X \in V(G)} X \right) + R$, replace the partial sum $\sum_{i \in I} A_i + \sum_{j \in J} B_j$ with the sum $\sum_{i \in I} A_i + \Rem(e)$; this will preserve the set $S$. Then, replace $R$ with $R + \Rem(e)$.

\item Eliminate the edge $e$ and its input nodes from the graph. This operation leaves the graph in a valid state since the input nodes of $e$ had indegree $0$ (by our choice) and outdegree $1$ (by condition $3$ from Definition $3.1$). Also, by eliminating the nodes $B_j$ with $j \in J$ from $V(G)$, we have recovered the equality $S = \left( \sum_{X \in V(G)} X \right) + R$.
\end{enumerate}

We note that each loop preserves the properties of a reduction graph listed in Definition $3.1$ (only condition $3$ really needs to be checked, which is easy). Figure $6$ provides a minimalist illustration of this algorithm (the eliminated nodes and edges are marked in red):

\begin{figure}[h]
\centering
\includegraphics[scale = 0.9]{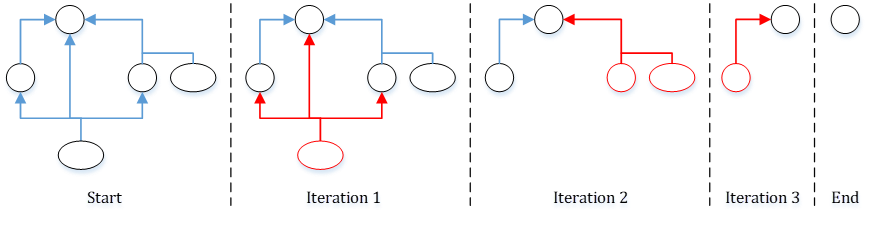}
\caption{Deconstruction of a Reduction Graph}
\end{figure}
\FloatBarrier

As suggested by the figure, we claim that we can iterate our algorithm as long as $E(G)$ is nonempty. Indeed, to pick the edge from step $1$, we can start by choosing a random edge $e \in E(G)$; if an input node $x$ of $e$ has nonzero indegree, replace $e$ with an edge that has $x$ as an output node, then repeat. This operation cannot proceed indefinitely since $G$ is acyclic and $E(G)$ is finite (if we reached the same edge twice, we would have found a cyclic path from an input node of that edge to itself). So there exists some $e \in E(G)$ whose input nodes all have indegree zero.

Therefore, our algorithm will only terminate when $E(G) = \emptyset$ (and this is bound to happen since there are finitely many edges, one of which is eliminated at each step). In consequence, once the algorithm terminates, all nodes that initially had outdegree $1$ must have been eliminated from $V(G)$ at step $3$ of some iteration, and the only remaining node is the root (which was never eliminated since it has outdegree $0$). Letting $R$ be the sum of all remainder sets from the initial graph, our invariant identity now reads:
\begin{equation}
S = \langle r(G) \rangle + R
\end{equation}

This gives us precisely the desired relation $(16)$, once we reconstruct the graph $G$.
\end{proof}

\begin{example}
The algorithm described above becomes more intuitive when visualized. Take, for instance, $V(G) = \{ \langle ab \rangle, \langle bc \rangle, \langle ca \rangle, \langle 2ab - ac \rangle, \langle ab + bc + ca \rangle \}$, where $a, b, c$ are pairwise relatively prime positive integers with $b < 2c$. The edges of the reduction graph $G$ are described in the figure below:
\begin{figure}[h]
\centering
\includegraphics[scale = 0.8]{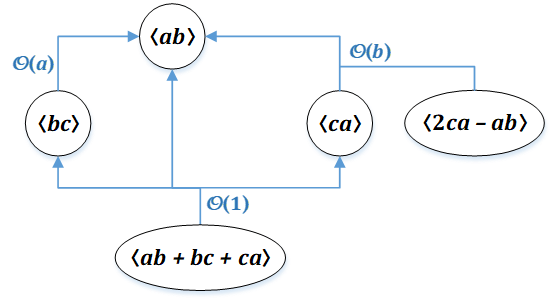}
\caption{Example of a Reduction Graph}
\end{figure}
\FloatBarrier

In the next section, we will develop the tools to check that each of these reductions is valid; for now, let us just assume that they work. Once can see that $r(G) = ab$, and $G$ contains three edges having weights $a$, $b$ and $1$; let us denote these edges by $e_a$, $e_b$, respectively $e_1$. Each edge corresponds to a reduction as described in Definition $2.1$, so there should exist finite remainder sets $\Rem(e_a)$, $\Rem(e_b)$, $\Rem(e_1)$ such that:
\begin{align}
&\langle ab \rangle + \langle bc \rangle = \langle ab \rangle \oplus \Rem(e_a) \\
&\langle ab \rangle + \left(\langle ca \rangle + \langle 2ca - ab \rangle\right) = \langle ab \rangle \oplus \Rem(e_b) \\
&\left(\langle ab \rangle + \langle bc \rangle + \langle ca \rangle\right) + \langle ab + bc + ca \rangle = (\langle ab \rangle + \langle bc \rangle + \langle ca \rangle) \oplus \Rem(e_1)
\end{align}

Given the weights of these edges, we also know that:
\begin{align}
&|\Rem(e_a)| = w(e_a) = a\\
&|\Rem(e_b)| = w(e_b) = b \\
&|\Rem(e_1)| = w(e_1) = 1
\end{align}

In particular, since any remainder set must contain $0$, we can infer from $(24)$ that $\Rem(e_1)$ should be the singleton $\{0\}$. Finding $\Rem(e_a)$ and $\Rem(e_b)$ will be easy using the results of the next section, but their exact forms are irrelevant for this example.

The arrangement of the edges in Figure $7$ gives two possible chronological orderings of the reductions: we should apply either $(21)$ then $(20)$ then $(19)$, or $(21)$ then $(19)$ then $(20)$. Suppose WLOG that we choose the first ordering; then the reasoning illustrated by the paths in Figure $7$ (which corresponds to the reduction algorithm shown in Figure $6$) is that:
\begin{equation}
\begin{split}
S(G) \xlongequal{\text{def}}&\q \langle ab \rangle + \langle bc \rangle + \langle ca \rangle + \langle 2ca - ab \rangle + \langle ab + bc + ca \rangle \\
\xlongequal{(21)}&\q \langle ab \rangle + \langle bc \rangle + \langle ca \rangle + \langle 2ca - ab \rangle + \Rem(e_1)\\
\xlongequal{(20)}&\q \langle ab \rangle + \langle bc \rangle + \Rem(e_b) + \Rem(e_1) \\
\xlongequal{(19)}&\q \langle ab \rangle + \Rem(e_a) + \Rem(e_b) + \Rem(e_1),
\end{split}
\end{equation}

which is the sum between the root $\langle ab \rangle$ and a finite set, as in relation $(16)$.  As we shall see soon, we can replace the sums above with direct sums provided that our reduction graph is \emph{total}:
\end{example}

\begin{definition}[Balance and Totality]
The \emph{balance} of a reduction graph $G$ is defined as:
\begin{equation}
\Bal(G) := \frac{r(G)}{\prod_{e \in E(G)} w(e)}
\end{equation}

We say that $G$ is a \emph{total} reduction graph if $\Bal(G) = \gcd(S(G))$. We note that the fraction above is well-defined, since $0$ must belong to each set $\Rem(e)$, whence $w(e) = |\Rem(e)| \geq 1$. 
\end{definition}

\begin{lemma}
For any reduction graph $G$ of a \emph{numerical} semigroup $S$, one has:
\begin{equation}
\Bal(G) \leq 1
\end{equation}
\end{lemma}

\begin{proof}
By Proposition $3.1$, we know that:
\begin{equation}
S = \langle r(G) \rangle + \sum_{e \in E(G)} \Rem(e)
\end{equation}

Since $S$ contains all sufficiently large integers, the relation above implies that $\sum_{e \in E(G)} \Rem(e)$ must contain at least one integer from each residue class modulo $r(G)$. In particular:
\begin{equation}
\left| \sum_{e \in E(G)} \Rem(e) \right| \geq r(G)
\end{equation}

Using the fact that the cardinality of a sum of sets is at most the product of their cardinalities (since the function $f(a, b) = a + b$ is a surjection from $A \times B$ onto $A + B$), we get that:
\begin{equation}
\begin{split}
\Bal(G) = 
\frac{r(G)}{\prod_{e \in E(G)} |\Rem(e)|} \leq \frac{r(G)}{\left| \sum_{e \in E(G)} \Rem(e) \right|} \leq 1
\end{split}
\end{equation}
\end{proof}

\begin{remark}
By a simple scaling (a technique to be detailed in Section $5$), one can see that for any reduction graph $G$ of a semigroup, $\Bal(G) \leq \gcd(S(G))$. In other words, totality (the state in which $\Bal(G) = \gcd(S(G))$) is in some sense the best we can get out of a reduction graph (i.e., the greatest product of weights given the nodes). The reason we prefer to work with total reduction graphs is the following theorem:
\end{remark}

\begin{theorem}
If $G$ is a total reduction graph of a numerical semigroup $S$ (so $\Bal(G) = 1$), the following equality of sets holds:
\begin{equation}
\Ap(S, r(G)) = \bigoplus_{e \in E(G)} \Rem(e)
\end{equation}

If all edges of $G$ correspond to simple reductions whose remainder sets have generating functions with closed forms (like those from the next section), we say that $S$ is \emph{graph-solvable} via $G$.
\end{theorem}

\begin{remark}
Graph-solvable semigroups extend a wide class of semigroups for which it is easy to compute Frobenius numbers, called \emph{free semigroups} \cite{Telescopic}. As we shall see later, free semigroups are graph-solvable using only two types of basic reductions (or only one if we make a simplification), so they have very simple formulations in terms of our graphical approach (see Figure $21$).
\end{remark}

\begin{proof}
Suppose $\Bal(G) = 1$, so equality is reached in equation $(30)$. Therefore, we must have:
\begin{equation}
\prod_{e \in E(G)} |\Rem(e)| = \left| \sum_{e \in E(G)} \Rem(e) \right|
\end{equation}

Since these are all finite sets, the equality above implies that no two $|E(G)|$-tuples from the Cartesian product $\prod_{e \in E(G)}$ \Rem(e) have the same sum. In other words, the sum of the remainder sets is a direct sum:
\begin{equation}
\sum_{e \in E(G)} \Rem(e) = \bigoplus_{e \in E(G)} \Rem(e)
\end{equation}

In this case, equation $(28)$ becomes:
\begin{equation}
S = \langle r(G) \rangle + \bigoplus_{e \in E(G)} \Rem(e)
\end{equation}

In order to reach equality in $(30)$, we must also have equality in $(29)$, which means that $\bigoplus_{e \in E(G)} \Rem(e)$ contains exactly one integer from each class modulo $r(G)$. So the sum in $(34)$ is a direct sum, thus the entire relation is in fact a reduction with remainder set $\bigoplus_{e \in E(G)} \Rem(e)$:
\begin{equation}
S = \langle r(G) \rangle \oplus \bigoplus_{e \in E(G)} \Rem(e)
\end{equation}

Taking $a = r(G)$ in equation $(5)$, we obtain the desired equation $(31)$.
\end{proof}

\begin{remark}
Our goal whenever we construct a reduction graph will be to reach the scenario from Theorem $3.1$ (i.e., final balance $1$). This explains why we defined reductions using direct sums rather than usual sums: if at any point in the construction of a reduction graph we used a reduction that did not leave behind a direct sum, the direct sum in equation $(35)$ could not be reached. A similar reasoning shows that any sub-reduction-graph (a subcollection of nodes and edges satisfying Definition $3.1$) of a total reduction graph needs to be total as well.
\end{remark}

\begin{example}
The reduction graphs from Figures $3$, $4$ and $7$ have balance
$1$, so they are total: $
\frac{9}{3 \q\cdot\q 3} = \frac{20}{4 \q\cdot\q 5} = \frac{a^n}{a^n} = \frac{ab}{a \q\cdot\q b} = 1
$.
In particular, Figure $3$ shows that the numerical semigroup $\langle 9, 12, 15, 20 \rangle$ is graph-solvable in two distinct ways, using the same nodes but different edges. Also, semigroups generated by geometric sequences are graph-solvable due to Figure $4$. To compute the Ap\'ery sets, Frobenius numbers and other attributes of graph-solvable semigroups, we have the next corollary:
\end{example}

\begin{corollary}
Under the hypothesis of Theorem $3.1$, we can find formulae for all the parameters of numerical semigroups that we defined in Section $1$:
\begin{align}
G(S) &= \N \setminus \left( \langle r(G) \rangle \q\oplus \bigoplus_{e \in E(G)} \Rem(e) \right) \\
F(S) &= - r(G) + \sum_{e \in E(G)} \max \left(\Rem(e)\right) \\
H_S(X) &= \frac{1}{1-X^{r(G)}} \prod_{e \in E(G)} \Gen(\Rem(e); X) \\
g(S) &= \lim_{x \to 1} \left(\frac{1}{1-x} - H_S(x)\right) = \frac{1-r(G)}{2} + \sum_{e \in E(G)} \mu(\Rem(e)),
\end{align}

where $\mu(\Rem(e))$ denotes the arithmetic mean of the elements in $\Rem(e)$. More generally, for any $k \geq 0$, if we denote by $\partial$ the derivation $\partial f(x) := xf'(x)$, we have:
\begin{equation}
S_k(S) = \lim_{x \to 1} \partial^k\left(\frac{1}{1-x} - H_S(x)\right)
\end{equation}
\end{corollary}

\begin{proof}
The identity in $(36)$ follows immediately from $(34)$ and the definition of sets of gaps (see $(6)$). Relation $(37)$ is a direct consequence of $(7)$, using the simple fact that the maximum of a sum of sets is the sum of their maximums. Similarly, $(38)$ follows from $(34)$ and Lemma $1.1$.

To deduce the first equality in $(39)$ and more generally, $(40)$ (note that the case $k = 0$ gives $S_0(S) = g(S)$), we need the observation that for any finite set $A$ of nonnegative integers, one has:
\[
\begin{split}
\lim_{x \to 1} \partial^k\Gen(A; x) &= \lim_{x \to 1} \sum_{a \in A} \partial^k x^a \\
&= \lim_{x \to 1} \sum_{a \in A} a^k x^a = \sum_{a \in A} a^k
\end{split}
\]

Therefore, $(40)$ is a direct consequence of $(9)$ and $(11)$; we do not even need the conditions of Theorem $3.1$ to state it explicitly, but we need these conditions to compute $H_S(x)$. We note that we must take limits as $x \to 1$ rather than evaluate at $x = 1$ directly, because our expressions are rational functions with poles at $1$. It remains to prove the second equality in $(39)$; we do this using relation $(38)$, the hypothesis that $\Bal(G) = 1$, and L'H\^opital's rule, denoting $R_e(x) := \Gen(\Rem(e); x)$:
\begin{equation}
\begin{split}
\lim_{x \to 1} \left(\frac{1}{1-x} - H_S(x)\right) &= \lim_{x \to 1} \frac{1+x+\ldots+x^{r(G)-1} - \prod_{e \in E(G)} R_e(x)}{1-x^{r(G)}}
\\
&= \lim_{x \to 1} \frac{1 + 2x + \ldots + (r(G)-1)x^{r(G)-2} - \prod_{e \in E(G)} R_e(x) \sum_{e \in E(G)} \frac{R_e'(x)}{R_e(x)}}{-r(G)x^{r(G)-1}}
\\
&= \frac{1 + 2 + \ldots + (r(G) - 1) - \prod_{e \in E(G)} w(e) \sum_{e \in E(G)} \mu(\Rem(e))}{-r(G)}
\\
&= \frac{1 - r(G)}{2} + \sum_{e \in E(G)} \mu(\Rem(e))
\end{split}
\end{equation}

More involved computations can of course be used to simplify the RHS of $(40)$, for any $k \geq 1$.
\end{proof}

\begin{corollary}
A numerical semigroup $S$, graph-solvable via $G$, is symmetric, respectively pseudo-symmetric, if and only if the sum $\sum_{e \in E(G)} \left(2\mu(\Rem(e)) - \max(\Rem(e))\right)$ equals $0$, respectively $1$. In these cases, we say that $G$ itself is symmetric, respectively pseudo-symmetric.
\end{corollary}

This corollary, which follows directly from $(37)$ and $(39)$, motivates the following definition:

\begin{definition}[Asymmetry]
We say that a reduction (edge) $e$ with remainder set $R$ has \emph{asymmetry} $A(e) := 2\mu(R) - \max(R)$, and that it is \emph{symmetric} or \emph{pseudo-symmetric} iff it has assymetry $0$, respectively $1$. The assymetry $A(G)$ of a reduction graph $G$ is defined as the sum of the assymetries of its edges, while the asymmetry of a numerical semigroup $S$ is defined as $A(S) := 2g(S) - F(S) - 1$.
\end{definition}

In light of $(37)$ and $(39)$, the asymmetry of a numerical semigroup equals the asymmetry of any reduction graph that describes it. In particular, a numerical semigroup that is graph-solvable using only symmetric edges is symmetric. Similarly, a numerical semigroup that is graph-solvable using only symmetric edges except for one pseudo-symmetric edge is pseudo-symmetric.

\begin{remark}
Computing the expressions from Corollary $3.1$ (especially the last one) can get quite convoluted, so we implemented a MAPLE program to help. All that the user needs to provide is a concise description of the edges of the reduction graph; we note that the program only works for graphs that use a fixed number of edges. Details are provided in the next implementation segment:
\end{remark}

\begin{implementation}
In the MAPLE script available at \cite{Program}, below the comment containing the phrase "LIST OF REDUCTIONS USED", the reader should write a representation of each edge used in their \emph{total} reduction graph of a numerical semigroup. Some edges (corresponding to the linear reductions, as we will see in the next section) need not be specified.

The representation of each edge should have the following structure:
\[
\text{Type(} \text{parameter}_1, \text{parameter}_2, \q\ldots\text{):}
\]

The possible types of edges and the parameters they require will be detailed in the next section. The script will print what the root of the reduction graph should be (for purposes of verification), followed by the Frobenius number, genus, assymetry and Hilbert series of the numerical semigroup. The reader also has the option to specify the value of a nonnegative integer $k$ towards the end of the script, which will result in computing the gaps' power sums $S_1, \ldots, S_k$. By default, we have set $k = 0$, since the computation of gaps' power sums may be very time-consuming.

We note that our program does not verify whether the reduction graph described by the user is well-defined (i.e., that it satisfies the conditions from Definition $3.1$); it cannot do so, because it has very limited information about the nodes of the graph, and incomplete information about its edges. The output of the program given inadequate input may vary from wrong results to runtime errors (e.g., division by $0$).
\end{implementation}

\begin{example}
For reasons to be explained in more detail in the next section, the edges of the reduction graph from Figure $7$ should be listed in our MAPLE script \cite{Program} in the following format:
\begin{align*}
&\text{Binary(a*b, b*c):} \\
&\text{Arithmetic(a*b, a*c - a*b, 2):}
\end{align*}

We note that only two instructions are needed because the third reduction (with input node $\langle ab + bc + ca \rangle$ and output nodes $\langle ab \rangle$, $\langle bc \rangle$, $\langle ca \rangle$) has a trivial remainder set of $\langle 0 \rangle$, which does not affect the overall Ap\'ery set given by Theorem $3.1$. This does not mean that the reduction is useless: its purpose is to connect the node $\langle ab + bc + ca \rangle$ to the rest of the graph.

Given the instructions above, the program will assume that the variables $a, b, c$ are pairwise relatively prime positive integers; any common divisor should be explicitly mentioned, e.g., by replacing $a$ with $ad$ and $b$ with $bd$. Running the script with this input produces the output below:
\begin{center}
\includegraphics[scale = 0.55]{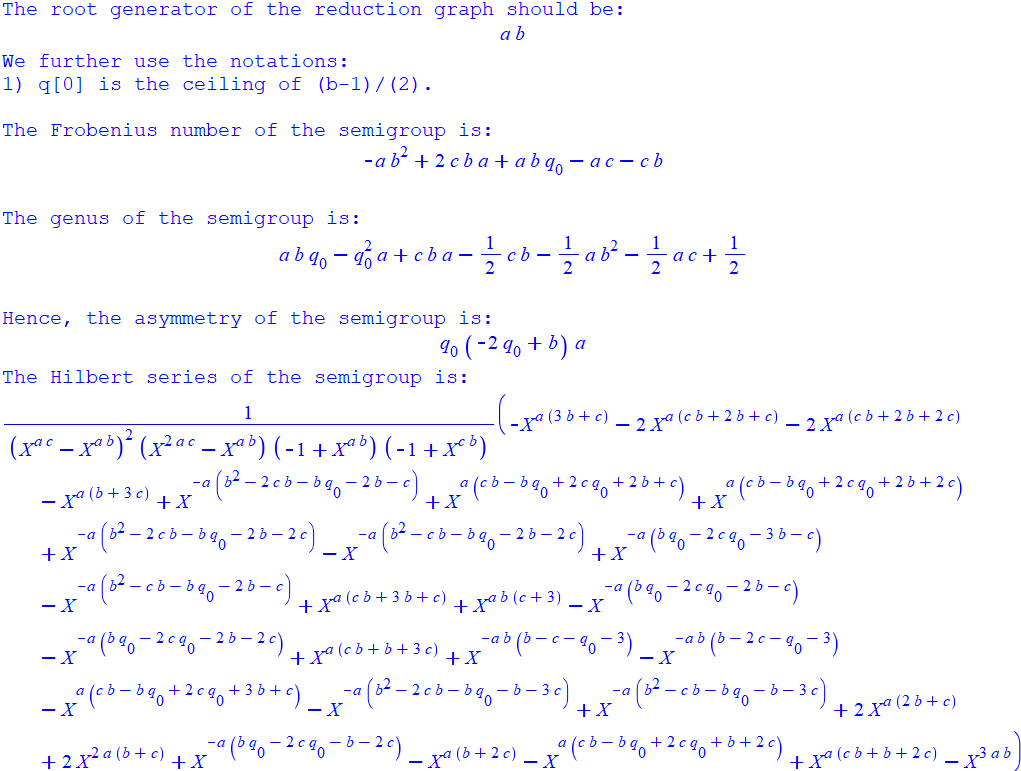}
\end{center}
\end{example}

\begin{remark}
When we characterize a numerical semigroup using a reduction graph, the decomposition of the initial Ap\'ery reduction into smaller and simpler edges must eventually come to a stop. The reductions that are simple enough that we cannot (or choose not to) decompose them into smaller edges are called \emph{basic reductions}, studied primarily in the next section and further developed in Section $5$. These include the reductions that our MAPLE script supports within its list of edges.
\end{remark}

\section{Basic Reductions and Linear Exchanges}

This section is dedicated to discovering the building bricks of our graphs: basic reductions, while the next sections will focus on transforming and combining these basic reductions to construct bigger reduction graphs. Of course, there is a trade-off between the complexity of the edges used in a reduction graph and the complexity of the graph's structure itself. Normally, if a reduction is known to hold and has a simple-to-describe remainder set, we might as well use it as a basic reduction to simplify our reasonings.

As a general intuition, a reduction graph with a complex structure indicates the existence of some \emph{multiplicative} relationship between the generators of the studied semigroup (e.g., geometric sequences). If the generators are instead related via \emph{additive} properties (e.g., arithmetic sequences), it is more likely that a basic reduction will be more useful. Hence when looking for basic reductions, it is preferred to develop techniques that exploit additive relationships between nodes. The technique that we developed to this purpose is based on the following definition, and a little notation from linear algebra:

\begin{definition}[Linear Exchange]
Suppose $\ba = (a_1, \ldots, a_n)$ is a fixed vector of integers (in the terminology of the Coin problem, these could be our coin denominations), and $X$ is a subset of $\mathbb{Z}^n$ (which would contain the possible vectors of coin frequencies). A \emph{linear exchange} applied to a vector $\bx \in X$ in terms of $\ba$ is a substitution $\bx \to \bx + \bu$, where $\bu$ is a vector such that $\ba^T\bu = \boldsymbol{0}$, and $\bx + \bu \in X$. Note that this substitution preserves the value of $\ba^T \bx$.
\end{definition}

\begin{remark}
The vectors $\bu$ that may lead to linear exchanges in terms of the vector $\ba$ lie in the kernel of the map $\ba^T : \mathbb{Z}^n \to \mathbb{Z}$.
\end{remark}

\begin{example}
Suppose that we want to study the numerical semigroup $\langle 4, 5, 6 \rangle$. We will take $X := \N^3$ as our space of possible coin frequences (since $\langle 4, 5, 6 \rangle = \{4x + 5y + 6z : (x, y, z) \in \N^3 \}$), and $\ba := (4, 5, 6)$ as our vector of semigroup generators (i.e. coin denominations). Searching for helpful linear exchanges, we note that $4 \cdot 1 + 5 \cdot (-2) + 6 \cdot 1 = 0$ and $4 \cdot 3 + 6 \cdot (-2) = 0$. Therefore, the vectors $\bu = (1, -2, 1)$ and $\bu' = (3, 0, -2)$ lead to linear exchanges for some vectors in $X$. We will soon see how these observations can lead to a reduction graph for the aforementioned semigroup.
\end{example}

\begin{remark}
The composition of more linear exchanges applied to a vector $\bx$ is still a linear exchange applied to $\bx$. The name of this process comes from the fact that in the substitution $\bx \to \bx + \bu$, some of the entries of $\bx$ increase and some decrease, but the overall weighted gain (where the weights are the entries of $\ba$) is zero - as in an exchange of currency between different coin denominations. We can manipulate this idea to produce basic reductions, using the following lemma:
\end{remark}

\begin{lemma}[Linear Exchange]
Fix a vector $\ba = (a_1, \ldots, a_n)$ of integers and a set $X \subset \mathbb{Z}^n$. Let $X'$ be a set obtained from $X$ by applying some linear exchange $\bx \to \bx + \bu_{\bx}$ to each vector $\bx \in X$; note that by Definition $4.1$, $X' \subset X$. Then, we have the following equality of sets:
\begin{equation}
\ba^TX := \{ \ba^T\bx : \bx \in X \} = \{ \ba^T\bx : \bx \in X' \} = \ba^TX'
\end{equation}

In other words, when we consider all vectors of the form $a_1 x_1 + \ldots + a_n x_n$ with $(x_1, \ldots, x_n) \in X$, we may assume that $(x_1, \ldots, x_n)$ is restricted to the subset $X'$. In practice, $X'$ will be much smaller than $X$, which will help us construct basic reductions.
\end{lemma}

\begin{proof}
The proof of this lemma follows immediately by definition:
\begin{equation}
\{ \ba^T \bx : \bx \in X' \} = \{ \ba^T (\bx + \bu_{\bx}) : \bx \in X \} = \{ \ba^T \bx : \bx \in X \}
\end{equation}
\end{proof}

\begin{example}
In continuation of Example $3.1$, let $X' := \N \times \{0, 1\} \times \N$. Note that for any $\bx \in X = \N^3$, we can apply the linear exchange $\bx \to \bx + (1, -2, 1)$ repeatedly, until $\bx \in X'$. Since the composition of several linear exchanges is a linear exchange, we have reached the scenario from Lemma $4.1$, so:
\begin{equation}
\begin{split}
\langle 4 \rangle + \langle 5 \rangle + \langle 6 \rangle &= \{ 4x_1 + 5x_2 + 6x_3 : x_1, x_2, x_3 \in \N\} \\
&= \{ 4x_1 + 5x_2 + 6x_3 : x_1, x_3 \in \N, x_2 \in \{0, 1\}\} \\
&= \left(\langle 4 \rangle + \langle 6 \rangle\right) + \{0, 5\}
\end{split}
\end{equation}

Judging based on parity, we can observe a direct sum in the RHS, which leads to the reduction $\left(\langle 4 \rangle + \langle 6 \rangle\right) + \langle 5 \rangle = \left(\langle 4 \rangle + \langle 6 \rangle\right) \oplus \{0, 5\}$ with remainder set $\{0, 5\}$ and weight $2$. Now that we have reduced the monogenic semigroup $\langle 5 \rangle$ with respect to $\langle 4 \rangle + \langle 6 \rangle$, we can focus on the remaining sum $\langle 4 \rangle + \langle 6 \rangle$. We take $X := \N^2$, $\ba = (4, 6)$, and observe the linear exchange given by $\bu = (3, -2)$. Letting $X' := \N \times \{0, 1\}$, we can apply the linear exchange $\bx \to \bx + (3, -2)$ to any vector $\bx \in X$ repeatedly, until $\bx \in X'$. As in $(44)$, this leads to the reduction $\langle 4 \rangle + \langle 6 \rangle = \langle 4 \rangle \oplus \{0, 6\}$ with remainder set $\{0, 6\}$ and weight $2$ (the direct sum follows considering residues modulo $4$). Based on the two basic reductions we have found, we may already build a total reduction graph:

\begin{figure}[h]
\centering
\includegraphics[scale = 0.9]{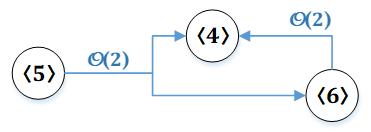}
\caption{Example of a Total Reduction Graph}
\FloatBarrier
\end{figure}

The root generator of this graph is $4$, so the Ap\'ery set $\Ap(\langle 4, 5, 6 \rangle, 4)$ should equal the sum of our remainder sets, i.e. $\{0, 5\} + \{0, 6\} = \{0, 5, 6, 11\}$. Surely, the semigroup generating system from this example is a particular case of an arithmetic progression; the more general case will be studied shortly.
\end{example}

Since applying linear exchanges from scratch can get a little tedious, we further provide a few corollaries of Lemma $4.1$, which we can use directly as building blocks of our reduction graphs:

\begin{corollary}[Linear Reduction]
Suppose that $a_1, \ldots, a_k, b$ are positive integers such that $b$ can be written as a linear combination of $a_1, \ldots, a_k$ with nonnegative integer coefficients. Then:
\begin{equation}
\langle a_1 \rangle + \ldots + \langle a_k \rangle + \langle b \rangle = \langle a_1 \rangle + \ldots + \langle a_k \rangle
\end{equation}

This is a reduction of the semigroup $\langle b \rangle$ with respect to $\langle a_1, \ldots, a_k\rangle$, with remainder set $\{0\}$, and hence \emph{symmetric} (note that $2\mu(\{0\}) - \max(\{0\}) = 0$). We could rewrite this equation as $\langle a_1, \ldots, a_k, b \rangle = \langle a_1, \ldots, a_k \rangle$, but the sums of monogenic semigroups are more closely related to the reduction graph representation. We illustrate this so-called \emph{linear reduction} as the edge below:

\begin{figure}[h]
\centering
\includegraphics[scale = 0.9]{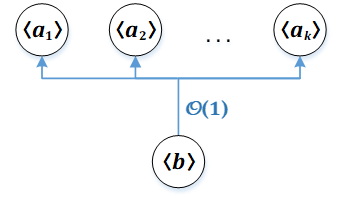}
\caption{Linear Reduction Edge}
\end{figure}
\FloatBarrier
\end{corollary}

\begin{proof}
Although this corollary is trivial, its proof serves as a good illustration of linear exchanges, whose applications will soon become more complicated. Write $\ba := (a_1, \ldots, a_k, b)$ and pick some $u_1, \ldots, u_k \geq 0$ such that $b = a_1 u_1 + \ldots + a_k u_k$. For each $\bx \in \N^{k+1}$, consider the vector:
\begin{equation}
\bu_{\bx} := (x_{k+1}u_1, \ldots, x_{k+1}u_k, -x_{k+1}),
\end{equation}

so that $\bx + \bu_{\bx}$ has the last entry equal to $0$. Then by Lemma $4.1$ and the linear exchanges $\bx \to \bx + \bu_{\bx}$ (note that $\ba^T \bu_{\bx} = 0$), we can assume that the last entry of $\bx$ is zero in the sense that:
\begin{equation}
\begin{split}
\langle a_1 \rangle + \ldots + \langle a_k \rangle + \langle b \rangle &= \{ \ba^T\bx : \bx \in \N^{k+1} \} \\
&= \{ \ba^T\bx : \bx \in \N^{k} \times \{0\} \} \\
&= \langle a_1 \rangle + \ldots + \langle a_k \rangle
\end{split}
\end{equation}

A more natural way to phrase this process is to consider only one linear exchange given by $\bu = (u_1, \ldots, u_k, -1)$; we can apply this exchange to each vector $\bx \in \N^{k+1}$ until the last entry of $\bx$ becomes $0$, which leads to the same judgment as in $(47)$.
\end{proof}

\begin{implementation}
Since linear reductions have trivial remainder sets, one need not specify them within the list of edges of the MAPLE script \cite{Program}.
\end{implementation}

\begin{remark}
Within the minimal system of generators of a semigroup, no generator $b$ can be written as a linear combination of other generators $a_1, \ldots, a_k$ (because otherwise the system could be made smaller by eliminating the generator $b$). Since we can always choose a minimal system of generators to describe a numerical semigroup, one may wonder whether linear reductions are of any good in practice. Here are two situations where these apparently trivial reductions play an essential role:

\begin{enumerate}
\item We may split a node $\langle b \rangle$ into a (direct) sum of two nodes:
\begin{equation}
\langle b \rangle = \langle ab \rangle \oplus \{ar : 0 \leq r < b\}
\end{equation}

While $b$ is probably not expressible as a linear combination of the other selected semigroup generators, $ab$ might be; this will allow us to reduce the node $\langle ab \rangle$ through a linear reduction, and further focus only on the leftover node $\{ar : 0 \leq r < b\}$. An application of this technique occurs for Fibonacci triplets (see Figure $11$).

\item It will sometimes be helpful to artificially add new monogenic semigroups to the list of nodes in a reduction graph, such that the new generators are linear combinations of the old generators. These artificial nodes (see Subsection $5.2$) will serve as bridges between the initial nodes of the graph, making use of linear reductions.
\end{enumerate}
\end{remark}

We now move on to another corollary of the linear exchange lemma:

\begin{corollary}[Residue Reduction]
Let $a > 0$ and $b$ be relatively prime integers (we allow $b < 0$). Consider a function $f: \N \to \N$ such that for all integers $0 \leq r < a$ and $q \geq 1$, one has:
\begin{equation}
f(aq + r) \geq f(r) - bq
\end{equation}

While this may seem like an arbitrary condition, it is a very common property for functions that show up in the Frobenius problem (e.g., any nondecreasing function works). Then we have the following reduction:
\begin{equation}
\begin{split}
\langle a \rangle + \{af(y) + by : y \geq 0\} &= \langle a \rangle \oplus \{af(r) + br : 0 \leq r < a \} \\
&= \langle a \rangle \oplus \bigO(a)
\end{split}
\end{equation}

This is called a \emph{residue reduction} modulo $a$; in particular, it is an Ap\'ery reduction, in the sense that its remainder set $\{f(r) : 0 \leq r < a \}$ is an Ap\'ery set with respect to $a$. Graphically, we can represent this reduction as the following edge:
\begin{figure}[h]
\centering
\includegraphics[scale = 0.9]{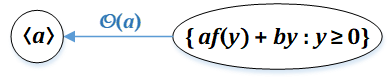}
\caption{Residue Reduction Edge}
\end{figure}
\FloatBarrier
\end{corollary}

\begin{remark}
We already knew that the LHS of $(50)$ was the direct sum between $\langle a \rangle$ and a set of cardinality $a$, due to Ap\'ery sets: recall relation $(5)$. The novelty here is that we can completely describe the Ap\'ery set $\{af(r) + br : 0 \leq r < a \}$ up to the nature of $f$ (in particular, we can probably find a closed form for its generating function and incorporate it into our MAPLE script, if $f$ is not too complicated). Hence it makes sense to use residue reductions as basic reductions.

In practice, we may not be so lucky to obtain a node of the form $\{af(y) + by : y \geq 0\}$ where $f$ has the property from $(49)$, but we may rather encounter a number of nodes whose sum is of this form. In that case, the reduction edge from Figure $10$ would have a split tail, similar to the one from Figure $3$. Such an edge can be found in Figure $11$.
\end{remark}

\begin{proof}
We note that the RHS of $(50)$ is a direct sum by considering residues modulo $a$ (using that $a$ and $b$ are relatively prime). Hence it suffices to show $(50)$ when the direct sum in the RHS is replaced with a usual sum. Define $\ba := (a, a, b)$, then $X := \{(x, f(y), y) : x, y \in \N\}$ and $X' := \{(x, f(r), r) : x \in \N, 0 \leq r < a\}$. With these notations, our claim in $(50)$ becomes:
\begin{equation}
\ba^T X = \ba^T X'
\end{equation}

By Lemma $4.1$, it suffices to show that for any $(x, f(y), y) \in X$, there is a linear exchange that maps it to a vector in $X'$. Fix such an $\bx := (x, f(y), y)$, and let $y = aq + r$ with $0 \leq r < a$. If $q = 0$, then $\bx$ is already in $X'$, so assume $q \geq 1$. Consider the linear exchange:
\begin{equation}
\begin{split}
\bx \to \bx + (&bq + f(y) - f(r), f(r) - f(y), -aq) \\
= (x +\q &bq + f(y) - f(r), f(r), r)
\end{split}
\end{equation}

Indeed, one can see that $\ba^T(bq + f(y) - f(r), f(r) - f(y), -aq) = 0$, so this is a valid linear exchange in terms of $\ba$ provided that the RHS of $(52)$ belongs to $X'$. By the definition of $X'$, the only thing we need to check is that $x + bq + f(y) - f(r) \geq 0$, which is true since $x \geq 0$ and $f(y) \geq f(r) - bq$ (this is where we use $(49)$). Our proof is now complete.
\end{proof}

\begin{remark}
Relations like $(50)$ are abundant in literature (although other authors may use different notation), especially in computing Ap\'ery sets with respect to some $a$ by studying each residue class modulo $a$ separately \cite{Frobenius, Arithm, Modified, Fibo}. Let us look at the simplest application of Corollary $4.2$:
\end{remark}

\begin{corollary}[Binary Reduction]
Given relatively prime positive integers $a$ and $b$, one has:
\begin{equation}
\begin{split}
\langle a \rangle + \langle b \rangle &= \langle a \rangle \oplus \{0, b, 2b, \ldots, (a-1)b\} \\
&= \langle a \rangle \oplus \bigO(a)
\end{split}
\end{equation}

This is a reduction with remainder set $\{0, b, 2b, \ldots, (a-1)b\}$, called a \emph{binary reduction}, whose weight is $a$. We note that we anticipated this reduction in Figure $1$ (right) as well as in equation $(15)$. It is easily checked that binary reductions are \emph{symmetric}, since $2(b + \ldots + (a-1)b) = a(a-1)b$.
\end{corollary}

\begin{proof}
Take $f$ to be the constant zero function in Corollary $4.2$.
\end{proof}

\begin{implementation}
In the MAPLE script available at \cite{Program}, binary reductions should be represented in the following format (preserving the notation from Figure $10$):
\begin{equation}
\text{Binary(a, b):}
\end{equation}

Due to a simple transformation of reduction graphs called scaling (covered in the next section), binary reductions can support a more general situation, where the parameters are not coprime; see Figure $13$ for the representation of such an edge. The reader should specify the existence of a common divisor of the two parameters explicitly (e.g., Binary(a*d, b*d)); otherwise, the program will assume that all variables involved are pairwise relatively coprime.
\end{implementation}

Now, using only Corollaries $4.1$ and $4.2$, let us show how to construct a reduction graph for certain Fibonacci triplets (the Frobenius problem for semigroups generated by these triplets was originally solved in \cite{Fibo}):

\begin{proposition}[Fibonacci Triplets]
Define the Fibonacci sequence by the usual recursive relation $F_{n+2} := F_n + F_{n+1}$, starting from $F_0 := 0$ and $F_1 := 1$. Let $i \geq 1, k \geq 3$ be integers. Then Figure $11$ shows a total reduction graph of the numerical semigroup $\langle F_i, F_{i+2}, F_k\rangle$: 
\begin{figure}[h]
\centering
\includegraphics[scale = 0.9]{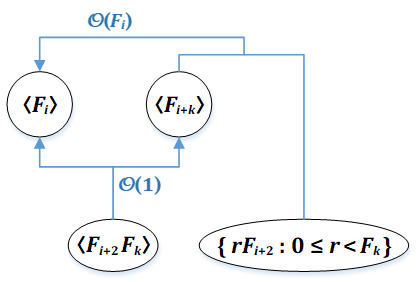}
\caption{Reduction Graph for Fibonacci Triplets}
\end{figure}
\FloatBarrier

In particular, the Frobenius number and genus of the semigroup $\langle F_i, F_{i+2}, F_k \rangle$ can be found using the graph above and Corollary $4.1$. These values were initially discovered in \cite{Fibo}; the tool of reduction graphs can help simplify the original reasoning.
\end{proposition}

\begin{proof}
Let $G$ denote the reduction graph from Figure $11$. Firstly, the sum of the nodes of $G$ is precisely the semigroup $\langle F_i, F_{i+2}, F_k \rangle$, since $\langle F_{i+2}\rangle = \langle F_{i+2}F_k \rangle + \{rF_{i+2} : 0 \leq r < F_k \}$. Also, $\Bal(G) = F_i \cdot (1 \cdot F_i)^{-1} = 1$, so $G$ is total. Hence we should only detail the edges used. The edge with input $\langle F_{i+2}F_k \rangle$ is a linear edge (Corollary $4.1$), based on the following easy identity:
\begin{equation}
F_{i+2}F_k = F_{k-2}F_i + F_{i+k}
\end{equation}

The other edge represents a residue reduction (Corollary $4.2$) with a split tail, which uses that:
\begin{equation}
\begin{split}
\{ rF_{i+2} : 0 \leq r < F_k \} + \langle F_{i+k} \rangle &= \{ rF_{i+2} + n(F_{i+2}F_k - F_{k-2}F_i) : 0 \leq r < F_k, n \geq 0 \} \\
&= \left\{ yF_{i+2} - F_i F_{k-2}\lf \frac{y}{F_k} \rf : y \geq 0 \right\}
\end{split}
\end{equation}

Above, we applied the substitution $y := nF_k + r$. Hence by Corollary $4.2$, it suffices to check that the function $f : \N \to \N$, $f(y) := - F_{k-2}\lf \frac{y}{F_k} \rf$ satisfies relation $(49)$ for $a = F_i$ and $b = F_{i+2}$.

Indeed, given $0 \leq r < F_i$ and $q \geq 1$, we have the gross approximation:
\begin{equation}
-f(aq+r) = F_{k-2}\lf \frac{qF_i + r}{F_k} \rf \leq F_{k-2}\frac{(q+1)F_i}{F_k} 
\end{equation}

Hence to satisfy $(49)$, it remains to prove that the RHS above is at most $qF_{i+2} - f(r)$, which is greater than or equal to $qF_{i+2}$. Thus it suffices to show that $(q+1)F_{k-2}F_i \leq qF_{i+2}F_k$, or equivalently, $F_{k-2}F_i \leq qF_{i+k}$ by $(55)$. Indeed, using that $q \geq 1$, we get:
\begin{equation}
\begin{split}
qF_{i+k} \geq F_{i+k} &= F_{i+2}F_k - F_{k-2}F_i \\
&\geq F_i (F_k - F_{k-2}) = F_i F_{k-1} \geq F_i F_{k-2}
\end{split}
\end{equation}
\end{proof}

Before we state the next application of linear exchanges, let us mention a useful substitution:

\begin{lemma}[Consecutive Differences]
If $a_1, \ldots, a_n$ are positive integers and $b_k := a_k - a_{k-1}$ for each $1 \leq k \leq n$ (where $a_0 := 0$), the following equality of sets holds:
\begin{equation}
\langle a_1 \rangle + \ldots + \langle a_n \rangle = \{ b_1 x_1 + \ldots + b_n x_n : x_1 \geq \ldots \geq x_n \geq 0 \}
\end{equation}
\end{lemma}

\begin{proof}
We have:
\begin{equation}
\begin{split}
\langle a_1 \rangle + \ldots + \langle a_n \rangle
&= \{ a_1 y_1 + \ldots + a_n y_n : y_1, \ldots, y_n \geq 0 \} \\
&= \{ b_1y_1 + \ldots + (b_1 + \ldots + b_n) y_n : y_1, \ldots, y_n \geq 0 \} \\
&= \{ b_1(y_1 + \ldots + y_n) + \ldots + b_n y_n : y_1, \ldots, y_n \geq 0 \} \\
&= \{ b_1x_1 + \ldots + b_n x_n : x_1 \geq \ldots \geq x_n \geq 0 \},
\end{split}
\end{equation}

where we used the substitution $x_k := y_k + \ldots + y_n$ for each $1 \leq k \leq n$.
\end{proof}

There are several applications of the consecutive differences substitution followed by linear exchanges; we only prove the most important one:

\begin{corollary}[Modified Arithmetic Reduction]
Let $h$, $a$, $k$ be positive integers and $d$ be a (not necessarily positive) integer such that $\gcd(a, d) = 1$ and $a + kd > 0$. Then we have the reduction:
\begin{equation}
\begin{split}
\langle a \rangle + \langle ha + d \rangle + \langle ha + 2d \rangle + \ldots + \langle ha + kd \rangle &= \langle a \rangle \oplus \left\{ ha\lc \frac{r}{k} \rc + dr : 0 \leq r < a \right\} \\
&= \langle a \rangle \oplus \bigO(a)
\end{split}
\end{equation}

We call this a \emph{modified arithmetic reduction}, according to the terminology from \cite{Modified}:
\begin{figure}[h]
\centering
\includegraphics[scale = 0.9]{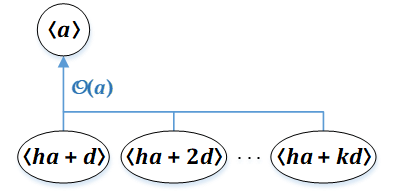}
\caption{Modified Arithmetic Reduction Edge}
\end{figure}
\FloatBarrier

We can use this Ap\'ery reduction to compute the Frobenius number and genus of semigroups generated by so-called \emph{modified arithmetic sequences} (i.e., of the form $\langle a, ha+d, ha+2d, \ldots, ha+kd \rangle$). These values were first discovered by A. Tripathi \cite{Modified} for the case $d > 0$; allowing $d$ to be negative is a small original contribution of this paper.

When $h = 1$, one obtains a semigroup generated by an arithmetic sequence, corresponding to a plain \emph{arithmetic reduction}; this case \cite{Arithm} has been studied long before its modified version in \cite{Modified}. We note that both graphs in Figure $3$ use one binary reduction and one arithmetic reduction. 
\end{corollary}

\begin{proof}
Our plan is to reduce the LHS of $(61)$ to $\langle a \rangle + \left\{ ha\lc \frac{y}{k} \rc + dy : y \geq 0 \right\}$ using Lemma $4.2$, and then to apply a residue reduction. Indeed, applying Lemma $4.2$ for $ha + d, \ldots, ha + kd$ yields that:
\begin{equation}
\langle a \rangle + \langle ha + d \rangle + \ldots + \langle ha + kd \rangle = \{ a x_0 + ha x_1 + d(x_1 + \ldots + x_k) : x_0 \geq 0, x_1 \geq \ldots \geq x_k \geq 0 \}
\end{equation}

It can be checked that when $x_1$ is kept fixed and $x_2 \geq \ldots \geq x_k \geq 0$ are variables bounded by $x_1$, the sum $x_1 + \ldots + x_k$ spans all the integers $y$ such that $x_1 \leq y \leq kx_1$. Then the RHS of $(62)$ becomes:
\begin{equation}
\{ a (x_0 + h x_1) + dy : x_0 \geq 0, x_1 \leq y \leq kx_1 \}
\end{equation}

Let $x := x_0 + hx_1$. Then for a fixed value of $y \geq 0$, since $x_0$ can be any nonnegative integer, $x$ can assume any value greater than or equal to the smallest $hx_1$ with $x_1 \leq y \leq kx_1$. The smallest such $x_1$ is clearly $x_1 = \lc \frac{y}{k} \rc$ (using the ceiling notation), and hence the LHS of $(61)$ is the same as:
\begin{equation}
\left\{ ax + dy : x \geq h\lc \frac{y}{k} \rc, y \geq 0 \right\} = \langle a \rangle + \left\{ ha\lc \frac{y}{k} \rc + dy : y \geq 0 \right\}
\end{equation}

Now we are in the position to apply a residue reduction (Corollary $4.2$), provided that the function $f : \N \to \N$, $f(y) := ha\lc \frac{y}{k} \rc$, satisfies relation $(49)$ for $b = d$. Let $0 \leq r < a$ and $q \geq 1$. If $d \geq 0$, we get $f(aq + r) \geq f(r) \geq f(r) - dq$. Otherwise, we have $-d > 0$ and we want to show that:
\begin{equation}
h\lc \frac{aq + r}{k} \rc \geq h\lc \frac{r}{k} \rc + (-d)q,
\end{equation}

knowing from the hypothesis that $a + kd > 0$, so $\frac{aq}{k} > (-d)q$. Using the rough estimates $h \geq 1$ and $\lc x \rc - \lc y \rc \geq \lc x - y \rc - 1 \geq x - y - 1$ for all $x, y \in \mathbb{R}$, we find that:
\begin{equation}
\begin{split}
h\left(\lc \frac{aq + r}{k} \rc - \lc \frac{r}{k} \rc\right)
\geq \frac{aq}{k} - 1 > (-d)q - 1
\end{split}
\end{equation}

Since this is a strict inequality of integers, we can turn it into a non-strict inequality by dropping the $-1$. 
This proves relation $(65)$, so we can apply a residue reduction on the RHS of $(64)$. The result of the reduction is precisely the RHS of $(61)$ (where $y \geq 0$ is replaced by $0 \leq r < a$), which completes our proof.
\end{proof}

\begin{remark}
Taking a limit of the sets from $(61)$ when $k \to \infty$ (in particular, $\lc \frac{r}{k} \rc \to 1$ if $r > 0$), in the sense that $x \in \lim_{n \to \infty} A_n$ if and only if $x \in A_n$ for all sufficiently large $n$, one can extend Corollary $4.4$ to infinite modified arithmetic sequences provided that $d \geq 0$:
\begin{equation}
\begin{split}
\langle a \rangle + \langle ha + d \rangle + \langle ha + 2d \rangle + \ldots &= \langle a \rangle \oplus \left(\{0\} \cup \left\{ ha + dr : 0 < r < a \right\}\right) \\
&= \langle a \rangle \oplus \bigO(a)
\end{split}
\end{equation}
\end{remark}

\begin{implementation}
Within the MAPLE script available on arXiv \cite{Program}, (modified) arithmetic reductions should be represented in the following format, using the notation from Figure $12$:
\begin{equation}
\text{Arithmetic(a, d, k, h):}
\end{equation}

The last parameter is optional; when $h$ is not specified, the script will assume that $h = 1$, yielding the case of usual arithmetic progressions. To indicate a reduction corresponding to an infinite progression, as described in $(67)$, one should write the word \emph{infinity} instead of $k$.

Like for binary reductions, our MAPLE implementation \cite{Program} of arithmetic reductions allows for a more general case than the one described above: the parameters $a$ and $d$ need not be coprime. This more general reduction scales all the nodes in Figure $12$ by a common divisor, say $g$; see Figure $13$ for the resulting edge. In that case, any common divisor of $a$ and $d$ should be mentioned explicitly (e.g., "Arithmetic(ag, dg, k, h)"); the program will otherwise assume that $\gcd(a, d) = 1$.

\end{implementation}

\begin{remark}
Computing the generating function of the Ap\'ery set $\left\{ ha\lc \frac{r}{k} \rc + dr : 0 \leq r < a \right\}$ from $(61)$ is a little more complicated than in the case binary reductions, due to the presence of the ceiling function. The closed form of this generating function is integrated in our MAPLE script, and we shall not bother proving it here; it is obtained by summing over an additional variable $j := \lc \frac{r}{k} \rc$, which varies from $0$ to $q := \lc \frac{a-1}{k} \rc$. The $q$-notation is preserved in the results of the program, as shown in the output fragment from the end of Section $3$.
\end{remark}

The same method (i.e., a consecutive differences substitution followed by linear exchanges, perhaps through a residue reduction) may be applicable to studying sequences of Mersenne \cite{Mersenne}, repunit \cite{Repunit}, Thabit \cite{Thabit, Four} and Cunningham \cite{Four} numbers. We will not develop these cases here since we want to focus on building reduction graphs rather than finding new basic reductions, but we can provide a table summary. The case of arithmetic progressions (discussed before) is listed as a clarifying example:
\begin{center}
\begin{tabular}{ c | c | c | c }
Name & Semigroup Generators & Consecutive Differences & Reason for Lin. Exch.
\\ \hline 
\rule{0pt}{3ex}    
Arithmetic
& 
$a, a + d, \ldots, a + kd$
&
$a, d, \ldots, d$
&
$d - d = 0$
\\
\rule{0pt}{3ex} 
Mersenne
&
$2^n - 1, 2^{n+1} - 1, \ldots$
&
$2^n - 1, 2^n, 2^{n+1}, \ldots$
&
$2 \cdot 2^{n+i} - 2^{n+i+1} = 0$
\\
\rule{0pt}{3ex} 
Repunit
&
$\frac{b^n - 1}{b - 1}, \frac{b^{n+1}-1}{b-1}, \ldots$
&
$\frac{b^n - 1}{b - 1}, b^n, b^{n+1}, \ldots$
&
$b \cdot b^{n+i} - b^{n+i+1} = 0$
\\ & & & \\
\rule{0pt}{3ex}
Thabit 1
&
$(b+1)b^n - 1, \ldots$
&
$(b+1)b^n - 1, (b^2-1)b^n, \ldots$
&
$b(b^2 - 1)b^{n+i} \qquad\q$
\\
\rule{0pt}{3ex} 
Thabit 2
&
$(b+1)b^n + 1, \ldots$
&
$(b+1)b^n + 1, (b^2-1)b^n, \ldots$
&
$- (b^2-1)b^{n+i+1} = 0$
\\
$(b \neq 3k+1)$ & & &
\\
\rule{0pt}{3ex} 
Thabit Ext.
&
$(2^k-1)2^n - 1, \ldots$
&
$(2^k-1)2^n - 1, (2^k-1)2^n, \ldots$
&
$2(2^k-1)2^{n+i} \qquad\q$
\\
& & & $- (2^k - 1)2^{n+i+1} = 0$
\\ & & & \\
\rule{0pt}{3ex}
Cunningham
&
$b^n + 1, b^{n+1} + 1, \ldots$
&
$b^n + 1, (b-1)b^n, \ldots$
&
$b(b-1)b^{n+i+1} = 0$
\\
$(b \text{ even})$ & & & $- (b-1)b^{n+i+1} = 0$
\end{tabular}
\end{center}

What is common for these cases is that the consecutive differences of the semigroup generators are simpler - or better related to each other through linear exchanges - than the generators themselves. Moreover, all of these cases lead to Ap\'ery reductions (recall Figure $2$, right).

\section{Valid Operations on Reduction Graphs}

Here we study a few operations one can apply on a given reduction graph to obtain a new reduction graph. These operations suggest a constructive approach to the Frobenius problem, in which one can discover a new graph-solvable semigroup by altering or combining previous ones.

\subsection{Scaling and Composition}

\begin{definition}[Scaling Sets and Reduction Graphs]
Let $S$ be a set of nonnegative integers, $G$ a reduction graph, and $k$ a positive integer. Then we define $kS := \{ks : s \in S\}$ to be the \emph{scaling} of $S$ by $k$, and $kG$ to be the graph obtained by scaling all of the nodes of $G$ by $k$, while keeping the weights the same. Also, given an individual reduction edge, its scaling by $k$ is obtained by scaling all of its inputs and outputs by $k$.
\end{definition}

\begin{lemma}[Scaling]
If $G$ is a reduction graph and $k$ is a positive integer, then $kG$ is a reduction graph with $S(kG) = kS(G)$, $\Bal(kG) = k\Bal(G)$, and $A(kG) = kA(G)$. In particular, any scaling of a symmetric edge is symmetric. Moreover, if $G$ is total, then $kG$ is total.
\end{lemma}

\begin{proof}
First we show that $kG$ is a reduction graph. It suffices to check that all of the edges of $kG$ are valid reduction edges, since the other requirements of Definition $3.1$ are clearly inherited from $G$ to $kG$. Consider a reduction edge from $E(G)$ with weight $w$, outputs $A_1, A_2, \ldots$ and inputs $B_1, B_2, \ldots$ like in Figure $2$ (left). Then there exists a remainder set $R$ of cardinality $w$ such that:
\begin{equation}
\sum A_i + \sum B_j = \left(\sum A_i\right) \oplus R
\end{equation}

It should be clear from $(1)$ and $(2)$ that scaling distributes with respect to usual and direct sums (even in the infinite cases), so we can scale everything by $k$ to get that:
\begin{equation}
\sum kA_i + \sum kB_j = \left(\sum kA_i\right) \oplus kR
\end{equation}

This creates a new reduction with the same weight $|kR| = |R| = w$, represented by the corresponding edge in $kG$. Hence $G$ is a reduction graph. The claims about the set $S(kG)$, the balance $\Bal(kG)$ and the asymmetry $A(kG)$ follow easily from their definitions. Finally, if $G$ is total, $\gcd(S(kG)) = \gcd(kS(G)) = k\gcd(S(G)) = k\Bal(G) = \Bal(kG)$, so $kG$ is also total.
\end{proof}

\begin{example}
In Figures $1$ and $12$, the binary and modified arithmetic reduction edges require that $\gcd(a, b) = 1$, respectively $\gcd(a, d) = 1$. Scaling these edges by some positive integer $g$, then substituting $ag \to a$, $bg \to b$, $dg \to d$, leads to the following more general reduction edges:

\begin{figure}[h]
\centering
\includegraphics[scale = 0.85]{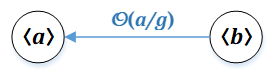}
\hspace{1cm}
\includegraphics[scale = 0.85]{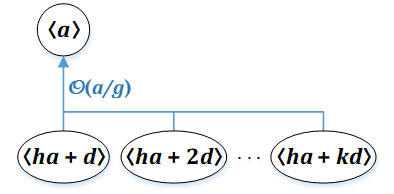}
\caption{Scaled Binary Reduction ($g = \gcd(a, b)$) vs. Scaled Arithmetic Reduction ($g = \gcd(a, d)$)}
\end{figure}
\FloatBarrier
\end{example}

\begin{remark}
Scaling is a reversible operation, in the sense that if all elements of all nodes from a reduction graph are divisible by some integer $k > 1$, one can simultaneously divide all of them by $k$. In this case, of course, the initial graph cannot describe a numerical semigroup, since all elements of the semigroup would share a nontrivial common divisor. Nevertheless, scaled reduction graphs and scaled edges can be useful in combination with other edges to create larger graphs that eventually describe numerical semigroups. For instance, even the simple binary reduction for $\gcd(a, b) = 1$ can now be further decomposed into two simpler edges:
\end{remark}

\begin{figure}[h]
\centering
\includegraphics[scale = 0.85]{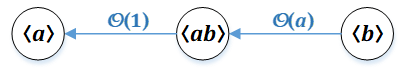}
\caption{Decomposition of a Binary Reduction Edge}
\end{figure}
\FloatBarrier

Above, we used a linear reduction from $\langle ab \rangle$ to $\langle a \rangle$ (which can also be seen as a scaled binary reduction), and a scaled binary reduction from $\langle b \rangle$ to $\langle ab \rangle$. The latter is just a scaling by $b$ of the binary edge from $\langle 1 \rangle$ to $\langle a \rangle$, which is much simpler:
\begin{equation}
\langle a \rangle + \langle 1 \rangle = \langle 1 \rangle = \langle a \rangle \oplus \{0, \ldots, a-1\}
\end{equation}

\begin{remark}
Under certain conditions, one can also apply scaling on a fragment of a reduction graph; we will refer to this operation as \emph{partial scaling}. Using the same reasoning as in the proof of Lemma $5.1$, one can scale any subset of vertices $V \subseteq V(G)$, as long as each edge affected by the scaling remains valid. For instance, scaling the following nodes by a positive integer $k$ does not affect the validity or type of the reduction edge involved:
\begin{itemize}
\item All inputs and simultaneously all outputs of any reduction edge (this is regular scaling);
\item All inputs of a (scaled) binary edge or (scaled) (modified) arithmetic edge, as long as $k$ is coprime with the output node (this follows by scaling $b$, $h$ and $d$ in Figure $13$);
\item The input node and some of the outputs of a linear reduction edge.
\end{itemize}
\end{remark}

\begin{example}
One can obtain the reduction graph for geometric sequences from Figure $4$ via a sequence of partial scalings starting from a much simpler chain, as shown in Figure $15$:
\begin{figure}[h]
\centering
\includegraphics[scale = 0.8]{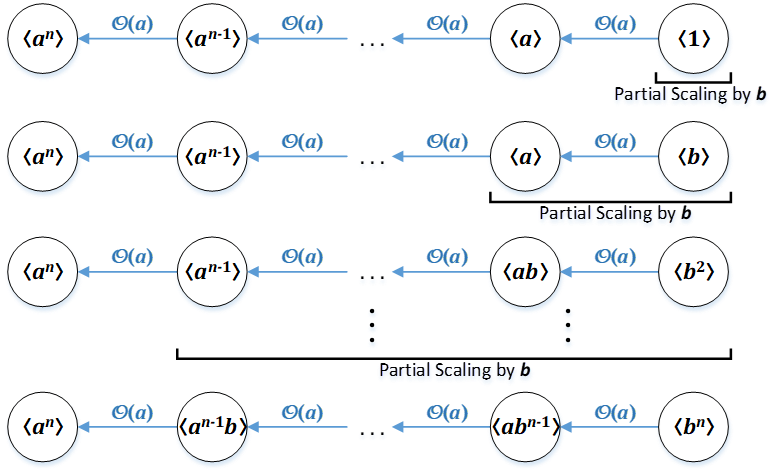}
\caption{Partial Scaling for Geometric Sequences}
\end{figure}
\FloatBarrier
\end{example}

To move on to more complex applications of scaling, we further define an operation between any two reduction graphs called \emph{composition}:

\begin{definition}[Composition]
Let $G$ and $G'$ be reduction graphs, and let $\langle a \rangle \in V(G)$ (we know that $G$ has at least one monogenic semigroup node: its root). Let $\langle a' \rangle$ be the root of $G'$. Then we define the \emph{composition} $G''$ of $G$ and $G'$ via the node $\langle a \rangle$ in two steps:
\begin{enumerate}
\item Construct the scaled graphs $a'G$ and $aG'$. In particular, the graph $a'G$ contains a node $\langle a'a \rangle$, which coincides with the root of $aG'$.
\item Identify the node $\langle a'a \rangle$ of $a'G$ with the root of $aG'$, and call the resulting graph $G''$.
\end{enumerate}

This process is illustrated in Figure $16$:

\begin{figure}[h]
\centering
\includegraphics[scale = 0.9]{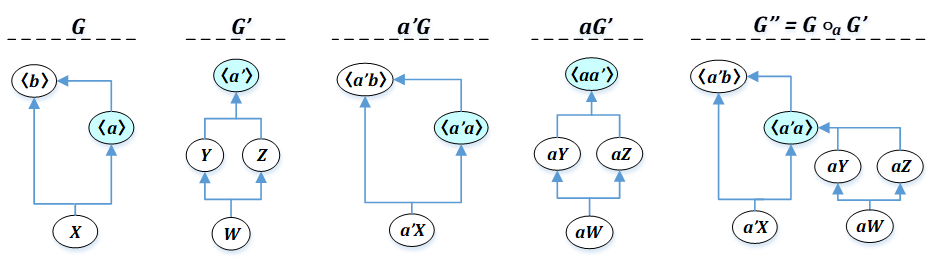}
\caption{Illustration of a Composition (weights not shown)}
\end{figure}
\FloatBarrier

As shown in the figure, we may write $G'' = G \circ_a G'$, as long as $G$ only contains one node of the form $\langle a \rangle$, to eliminate ambiguity. In practice, of course, there is rarely any reason to have the same node occur twice in a reduction graph. If the node $\langle a \rangle$ is the root of $G$, we can drop the subscript and write $G'' := G \circ G'$. We will refer to the operation "$\circ$" as \emph{composition by the root}.
\end{definition}

\begin{lemma}[Composition]
With the notations from Definition $5.2$, we claim that $G''$ is a reduction graph, $S(G'') = a' S(G) + a S(G')$, $\Bal(G'') = \Bal(G) \cdot \Bal(G')$, and $A(G'') = a'A(G) + aA(G')$. In particular, the composition of two symmetric reduction graphs is symmetric. Also, composition by the root is commutative and associative. 
\end{lemma}

\begin{proof}
By Lemma $5.1$, both $a'G$ and $aG'$ are valid reduction graphs, and identifying a node of $a'G$ with a node of $aG'$ does not affect the validity of their edges. Moreover, every node of $G''$ except for its root has outdegree $1$, since the root of $aG'$ either becomes the root of $G''$ or receives an outgoing edge. Clearly, $G''$ is also acyclic. The other conditions of Definition $3.1$ are trivially checked, hence $G''$ is a reduction graph.

Now by the way we constructed $G''$, we have $\{w(e) : e \in G''\} = \{w(e) : e \in G\} \biguplus \{w(e) : e \in G'\}$ and $V(G'') = V(G) \biguplus V(G') \setminus \{\langle a'a \rangle\}$ as multisets (where only one instance of $\langle a'a \rangle$ is eliminated to account for the repetition). Therefore:
\begin{equation}
\begin{split}
\Bal(G'') &= \frac{a'b}{\prod_{e \in E(G'')} w(e)} \\
&= \frac{b}{\prod_{e \in E(G)} w(e)} \cdot \frac{a'}{\prod_{e \in E(G')} w(e)} =  \Bal(G) \cdot \Bal(G')
\end{split}
\end{equation}

And since $\langle a'a \rangle = \langle a'a \rangle + \langle aa' \rangle$:
\begin{align}
S(G'') &= \sum_{X \in V(a'G)} X + \sum_{X \in V(aG')} X = a'S(G) + aS(G') \\
A(G'') &= \sum_{e \in E(a'G)} A(e) + \sum_{e \in E(aG')} A(e) = a'A(G) + aA(G')
\end{align}

The commutativity and associativity of composition by the root follow easily from the commutativity and associativity of multiplication, and from the symmetry between the inputs of a reduction edge (i.e., there is no preferred order of these inputs).
\end{proof}

\begin{remark}
The reduction graph for \emph{composed} geometric sequences (Figure $5$) is, as expected, the composition by the root of two reduction graphs for geometric sequences (Figure $4$). In fact, the reduction graph for geometric sequences can itself be seen as a composition of $n$ binary reductions:
\[
G = B \circ_b (\ldots \circ_b (B \circ_b (B \circ_b B))),
\]

where $G$ denotes the reduction graph from Figure $4$ and $B$ denotes the reduction graph from Figure $1$ (right). To clarify, the equation above contains $n$ instances of the graph $B$. In particular, since $\Bal(B) = 1$, we have $\Bal(G) = \Bal(B)^n = 1$, hence we can apply Theorem $3.1$ and Corollary $3.1$ to characterize the numerical semigroups generated by geometric sequences. By generalizing this idea, we obtain the case of so-called \emph{compound} sequences \cite{Compound}:
\end{remark}

\begin{example}[Compound Sequences]
Let $k, a_1, \ldots, a_k, b_1, \ldots, b_k$ be positive integers such that $\gcd(a_i, b_j) = 1$ for all $k \geq i \geq j \geq 1$. For all $0 \leq i \leq k$, define: $n_i := a_{i+1}a_{i+2}\ldots a_k \cdot b_1 b_2\ldots b_i$.

Of course, $n_0, \ldots, n_k$ form a geometric progression when $a_1 = \ldots = a_k$ and $b_1 = \ldots = b_k$. One can check that $\gcd(n_i, n_{i+1}) = \frac{n_i}{a_{i+1}}$ and $\gcd(n_0, \ldots, n_k) = 1$, so that $\langle n_0, \ldots, n_k \rangle$ is a numerical semigroup. Using $k$ scaled binary reductions, we can build a total reduction graph of this semigroup:

\begin{figure}[h]
\centering
\includegraphics[scale = 0.85]{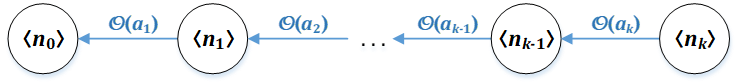}
\caption{Reduction Graph for a Compound Sequence}
\end{figure}
\FloatBarrier

The reduction graph from Figure $17$, call it $G$, can be seen as the composition of $k$ different binary reductions. Indeed, if $B_i$ denotes the binary reduction graph from $b_i$ to $a_i$, for $1 \leq i \leq k$, one can see that:
\begin{equation}
G = B_1 \circ_{b_1} (\ldots \circ_{b_{k-3}} (B_{k-2} \circ_{b_{k-2}} (B_{k-1} \circ_{b_{k-1}} B_k))),
\end{equation}

In particular, $\Bal(G) = 1^k = 1$, so the Ap\'ery set, Frobenius number and genus of semigroups generated by compound sequences can be easily computed using Theorem $3.1$ and Corollary $3.1$; the query for these values and other semigroup attributes was the subject of the paper in \cite{Compound} (surely, the authors of the referenced paper used different methods).
\end{example}

\begin{remarks}
In the original paper \cite{Compound}, the authors impose some additional restrictions on compound sequences (i.e., $2 \leq a_i < b_i$ for each $1 \leq i \leq k$), but we do not find these necessary for the purpose of constructing a reduction graph. Also, the graph in Figure $17$ can be alternatively achieved through a sequence of partial scalings, by generalizing the process from Figure $15$.
\end{remarks}

Another case when composition is useful for constructing reduction graphs concerns semigroups with three special generators (originally studied in \cite{Pakornrat}):

\begin{example}[Special Case for $3$ Generators]
Let $a, b, c$ be positive integers such that $c \divides \lcm(a, b)$ and $\gcd(a, b, c) = 1$. Let $g_1 := \gcd(a, c)$, $g_2 := \gcd(b, c)$ and note that $\gcd(g_1, g_2) = 1$. Then the numerical semigroup $\langle a, b, c \rangle$ has the following reduction graph built from two scaled binary edges:

\begin{figure}[h]
\centering
\includegraphics[scale = 0.8]{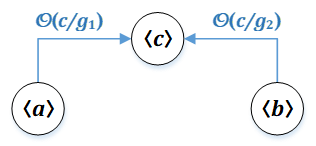}
\caption{Reduction Graph for a Special Triplet}
\end{figure}
\FloatBarrier

Note that $c \divides \gcd(\lcm(a, b), c) = \lcm(g_1, g_2) \divides c$. Therefore, $c = \lcm(g_1, g_2) = g_1g_2$ since $g_1$ and $g_2$ are coprime. Thus the products of the weights of the reduction graph from Figure $17$, call it $G$, is $\frac{c^2}{g_1 g_2} = c = r(G)$. Hence $G$ is a total reduction graph, suitable for Theorem $3.1$ and Corollary $3.1$. The Frobenius number found this way coincides with the original result from $\cite{Pakornrat}$.

Breaking down this special case reveals the simplest possible composition of reduction graphs: a composition by the root of two binary reductions. Indeed, let $a = xg_1$, $b = yg_2$, and let $B_1$, $B_2$ denote the binary reduction graphs from $x$ to $g_2$, respectively $y$ to $g_1$ (note that $\gcd(x, g_2) \divides \gcd(a, c) = g_1$ and $\gcd(y, g_1) \divides \gcd(b, c) = g_2$, so $\gcd(x, g_2) = \gcd(y, g_1) = 1$). In this case, one has $G = B_1 \circ B_2$.
\end{example}

\subsection{Artificial Nodes and Enrichment}

The operations presented in this subsection add new nodes to the structure of a reduction graph:

\begin{definition}[Artificial Node]
Let $V$ be a set of reduction graph nodes, $E$ a set of reduction edges on $V$, and $G := (V, E)$. Suppose that $\langle a_1 \rangle, \langle a_2 \rangle, \ldots, \langle a_n \rangle \in V$, and $b$ is a linear combination of $a_1, \ldots, a_n$ with nonnegative integer coefficients (i.e., $b \in \langle a_1, \ldots, a_n \rangle$). Suppose $G'$ is a reduction graph with $V(G') := V \bigcup \{\langle b \rangle\}$ and $E(G') = E \bigcup E'$, where $E'$ is a collection of edges that connect $\langle b \rangle$ to $G$. Then we say that $\langle b \rangle$ is an \emph{artificial node} added to $G$ in order to create $G'$.
\end{definition}

\begin{remark}
Informally, $G$ may be seen as a possibly incomplete reduction graph (or a reduction graph in the making). The semigroups $S(G)$ (defined as $\sum_{X \in V} X$) and $S(G')$ are equal, since $\langle a_1 \rangle + \ldots + \langle a_n \rangle + \langle b \rangle = \langle a_1 \rangle + \ldots + \langle a_n \rangle$ (this identity is precisely a linear reduction). Therefore, adding artificial nodes can be a useful step in constructing a reduction graph for a studied numerical semigroup, as they raise the possibility of forming more connections between nodes. There are two important corollaries of this method:
\end{remark}

\begin{corollary}[Recursive Formulae of Brauer and Shockley \cite{Frobenius}]
Let $a_1, \ldots, a_n$ be positive integers and $d \divides \gcd(a_1, \ldots, a_{n-1})$. Then one has:
\begin{align}
F(a_1, \ldots, a_n) &= dF\left( 
\frac{a_1}{d}, \ldots, \frac{a_{n-1}}{d}, a_n \right) + a_n(d-1)\\
\nonumber
g(a_1, \ldots, a_n) &= dg\left( 
\frac{a_1}{d}, \ldots, \frac{a_{n-1}}{d}, a_n \right) + \frac{(a_n-1)(d-1)}{2}\\
\nonumber
H_{\langle a_1, \ldots, a_n \rangle}(X) &= H_{\left\langle 
\frac{a_1}{d}, \ldots, \frac{a_{n-1}}{d}, a_n \right\rangle}\left(X^d\right) \cdot \frac{X^{da_n}-1}{X^{a_n}-1}
\end{align}
\end{corollary}

\begin{proof}
Since $a_1, \ldots, a_n$ are relatively prime, so are $\frac{a_1}{d}, \ldots, \frac{a_{n-1}}{d}$ and $a_n$, thus the latter numbers generate a numerical semigroup $S$. Then $S$ must have some Ap\'ery set in terms of $a_n$, call it $A := \Ap(S, a)$, which is the remainder set of the Ap\'ery reduction edge in Figure $19$ (left):

\begin{figure}[h]
\centering
\includegraphics[scale = 0.85]{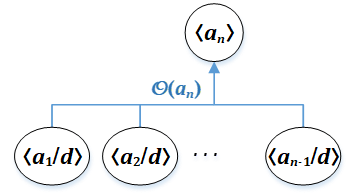}
\hspace{1cm}
\includegraphics[scale = 0.85]{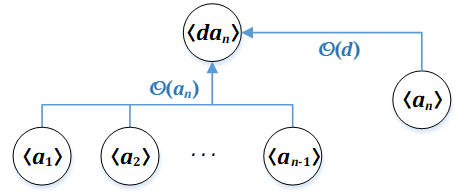}
\caption{An Ap\'ery Reduction (left) vs. Recursive Reduction Graph (right)}
\end{figure}
\FloatBarrier

Consequently, the remainder set of the corresponding scaled edge in Figure $19$ (right) is $dA$; the other edge is just a scaled binary reduction with remainder set $a_n\{0, 1, \ldots, d-1\}$, due to $(15)$. We note that the graph on the right is total, and describes the numerical semigroup $\langle a_1, \ldots, a_n \rangle$ by adding the artificial node $\langle da_n \rangle$ (which is a linear combination of $a_n$ alone). In particular, this graph is precisely the composition by the root of the graph on the left and the binary edge $\langle d \rangle \gets \langle 1 \rangle$. The results now follow easily by applying relations $(37)$, $(38)$ and $(39)$ to both graphs in Figure $19$ and phrasing everything in terms of $a_1, \ldots, a_n, d$ and $A$.
\end{proof}

\begin{remark}
Corollary $5.1$ works for any common divisor $d$ of $a_1, \ldots, a_{n-1}$, although these results are often phrased only for $d = \gcd(a_1, \ldots, a_{n-1})$ \cite{Triangular, Tripathi}. In particular, $d = 1$ gives trivial equalities.
\end{remark}

\begin{corollary}[Linear-Binary Reduction]
Let $k$, $a_1, a_2, \ldots, a_k$, and $b$ be positive integers and let $c := \frac{\gcd(a_1, \ldots, a_k)}{\gcd(a_1, \ldots, a_k, b)}$. Suppose that $bc \in \langle a_1, \ldots, a_k \rangle$. Then we can combine the scaled binary edge from $\langle b \rangle$ to $\langle bc \rangle$ (see $(71)$) with the linear edge from $\langle bc \rangle$ to $\langle a_1 \rangle, \langle a_2 \rangle, \ldots, \langle a_k \rangle$ in order to form a so-called \emph{binary-linear reduction edge} from $\langle b \rangle$ to $\langle a_1 \rangle, \langle a_2 \rangle, \ldots, \langle a_k \rangle$. More precisely, to the initial configuration of nodes $\langle a_1 \rangle, \ldots, \langle a_k \rangle, \langle b \rangle$, one can add an artificial node $\langle bc \rangle$ since $bc \in \langle b \rangle$, then connect it to the other nodes, and then eliminate it by merging the two edges into one:

\begin{figure}[h]
\centering
\includegraphics[scale = 0.78]{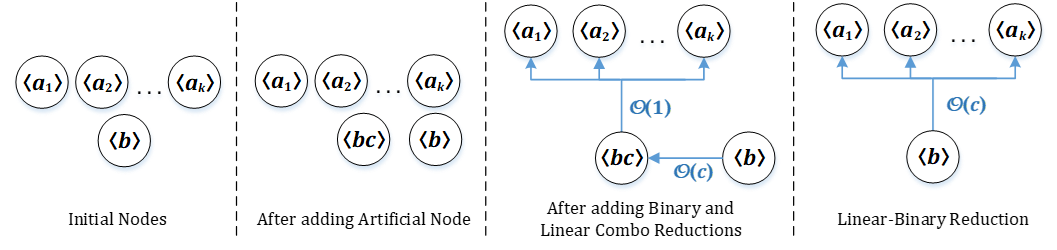}
\caption{Origin of the Linear-Binary Reduction, with $c := \frac{\gcd(a_1, \ldots, a_k)}{\gcd(a_1, \ldots, a_k, b)}$, $bc \in \langle a_1, \ldots, a_k \rangle$}
\end{figure}
\FloatBarrier

\begin{proof}
Formally, by combining the binary reduction $\langle bc \rangle + \langle b \rangle = \langle bc \rangle \oplus \bigO(c)$ with the linear reduction $\langle a_1 \rangle + \ldots + \langle a_n \rangle + \langle bc \rangle = \langle a_1 \rangle + \ldots + \langle a_n \rangle$, we obtain the \emph{linear-binary} reduction:
\begin{equation}
\begin{split}
\langle a_1 \rangle + \ldots + \langle a_k \rangle + \langle b \rangle
&= \langle a_1 \rangle + \ldots + \langle a_k \rangle + \langle bc \rangle + \langle b \rangle \\
&= \langle a_1 \rangle + \ldots + \langle a_k \rangle + \langle bc \rangle + \bigO(c) \\
&= \langle a_1 \rangle + \ldots + \langle a_k \rangle + \bigO(c) \\
&= \left(\langle a_1 \rangle + \ldots + \langle a_k \rangle\right) \oplus \bigO(c)
\end{split}
\end{equation}

It remains to motivate the direct sum in the RHS of $(77)$. Note that the remainder $\bigO(c)$ from the reduction $\langle bc \rangle + \langle b \rangle = \langle bc \rangle \oplus \bigO(c)$ represents the set $\{br : 0 \leq r < c \}$. On the other hand, $c$ is by definition coprime with $b$, so $\{br : 0 \leq r < c \}$ attains each residue class modulo $c$ exactly once. Since $c \divides a_1, \ldots, a_k$, there will be no overlap in the sum $\langle a_1, \ldots, a_k \rangle + \{br : 0 \leq r < c \}$, so the latter is a direct sum.
\end{proof}

\end{corollary}

\begin{remark}
A numerical semigroup is graph-solvable using only linear and (scaled) binary reductions if and only if it is \emph{free} \cite{Telescopic}. A numerical semigroup is called free iff it is generated by a so-called \emph{telescopic sequence} $a_1, \ldots, a_n$, defined by the property:
\begin{equation}
\frac{a_{k+1}}{\gcd(a_1, \ldots, a_{k+1})} \in \left\langle \frac{a_1}{\gcd(a_1, \ldots, a_{k})}, \ldots, \frac{a_k}{\gcd(a_1, \ldots, a_k)} \right\rangle,
\end{equation}

for each $1 \leq k < n$. Indeed, letting $b := a_{k+1}$ and $c := \frac{\gcd(a_1, \ldots, a_k)}{\gcd(a_1, \ldots, a_k, b)}$, the condition in $(78)$ is equivalent to $bc \in \langle a_1, \ldots, a_k \rangle$, which is the necessary condition for linear-binary reductions. Therefore, every telescopic sequence has a corresponding reduction graph which is simply a chain of linear-binary reductions (where $c_k := \frac{\gcd(a_1, \ldots, a_k)}{\gcd(a_1, \ldots, a_k, a_{k+1})}$):

\begin{figure}[h]
\centering
\includegraphics[scale = 0.95]{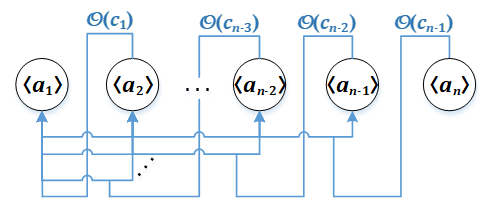}
\caption{Reduction Graph for Telescopic Sequences}
\end{figure}
\FloatBarrier

We note that the balance of this graph is $\frac{a_1}{c_2\ldots c_k} = a_1 \cdot \left(\frac{\gcd(a_1)}{\gcd(a_1, a_2)} \cdot \ldots \cdot \frac{\gcd(a_1, \ldots, a_{n-1})}{\gcd(a_1, \ldots, a_n)}\right)^{-1} = 1$, since $\gcd(a_1, \ldots, a_n) = 1$ by the fact that $\langle a_1, \ldots, a_n \rangle$ is a numerical semigroup. Hence free semigroups are graph-solvable using only linear and binary reduction edges, which are both symmetric, thus:
\end{remark}

\begin{corollary}
Free numerical semigroups are symmetric.
\end{corollary}

Conversely, if a numerical semigroup $S$ is graph-solvable using only linear and (scaled) binary reduction edges, then it is also graph-solvable using linear-binary reductions (since both linear and scaled binary reductions are particular cases of linear-binary reductions: take $k = 1$, respectively $c = 1$ in Figure $20$). Then by selecting, at each step, a node that has no outgoing path towards the previously selected nodes (which is possible since reduction graphs are acyclic), we have to end up with the picture from Figure $21$ (we may need to add a few outputs to the linear-binary reductions, but this is always allowed). Therefore, $S$ is free.

\begin{example}[Triangular and Tetrahedral numbers]
In \cite{Triangular}, the authors investigate semigroups generated by sequences of $3$ consecutive triangular numbers and $4$ consecutive tetrahedral numbers. Both semigroups turn out to be numerical and free after an analysis of 2, respectively 6 cases. In particular, the case of triangular numbers only requires (scaled) binary reductions:

\begin{figure}[h]
\centering
\includegraphics[scale = 0.85]{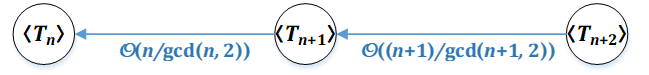}
\caption{Reduction Graph for Triplets of Triangular Numbers}
\end{figure}
\FloatBarrier

Above, we denoted $T_n := \frac{n(n+1)}{2}$ and we used that $\frac{T_n}{\gcd(T_n, T_{n+1})} = \frac{n}{\gcd(n, n+2)} = \frac{n}{\gcd(n, 2)}$. We note that the graph above has balance $1$ because $\gcd(n, 2) \cdot \gcd(n+1, 2) = 2 = \frac{T_n}{n(n+1)}$. Building a reduction graph for tetrahedral numbers, on the other hand, requires at least one linear reduction (and hence a linear-binary reduction); we exemplify this for the case $n \equiv 0\q (\text{mod } 6)$:

\begin{figure}[h]
\centering
\includegraphics[scale = 0.82]{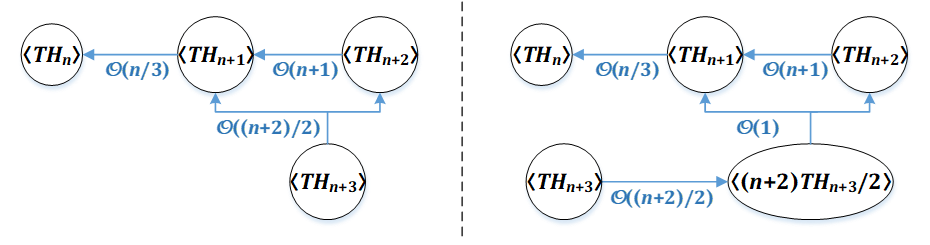}
\caption{Reduction Graphs for Quadruplets of Tetrahedral Numbers, $n \equiv 0\q (\text{mod } 6)$}
\end{figure}
\FloatBarrier

Here, we used the notation $\text{\emph{TH}}_{n} := \frac{n(n+1)(n+2)}{6}$ for the $n^{th}$ tetrahedral number. The graph on the right of Figure $23$ shows a more detailed version of the graph on the left, in which the linear-binary reduction edge is replaced by a linear and a (scaled) binary reduction edge. The weights of the scaled binary reductions used can be verified easily given that $n$ is divisible by $6$, and the linear reduction is motivated by the equality $
\frac{n+4}{2} \text{\emph{TH}}_{n+1} + 2\text{\emph{TH}}_{n+2} = \frac{n+2}{2} \text{\emph{TH}}_{n+3}$.
\end{example}

\begin{remark}
The graph on the right of Figure $23$ is more complex, but in some ways also more useful than the graph on the left. Indeed, by splitting the linear-binary edge into a linear one and a binary one, we open the possibility of \emph{enriching} the binary edge:
\end{remark}

\begin{definition}[Enrichment]
An \emph{enrichment} of a reduction graph $G$ is an extension of $V(G)$ by a set of nodes $V'$ and a replacement of a reduction edge $e \in E(G)$ by another edge $e'$ such that:
\begin{enumerate}
\item $w(e) = w(e')$;
\item All outputs of $e'$ are outputs of $e$, and all inputs of $e$ are inputs of $e'$;
\item $V'$ is the set of inputs of $e'$ that are not inputs of $e$.
\end{enumerate}
\end{definition}

\begin{lemma}
The result of enriching a reduction graph is a reduction graph with the same balance.
\end{lemma}

\begin{proof}
Preserve the notations from Definition $5.4$, and let $G'$ be the enriched graph. Supposing that $G'$ is a reduction graph, the fact that $\Bal(G) = \Bal(G')$ is clear since $r(G) = r(G')$ and $w(e) = w(e')$. It remains to show that $G'$ is a reduction graph, in particular that all nodes except for the root of $G'$ have outdegree equal to $1$. Firstly, the outdegrees of the nodes in $V(G)$, which include the root of $G'$, are not affected by the enrichment since all inputs of $e$ are inputs of $e'$. Secondly, the outdegree of each node in $V'$ is $1$, since it is only connected to the rest of the graph through $e'$. Lastly, $G'$ is acyclic because $G$ is acyclic and all outputs of $e'$ are outputs of $e$. The other conditions of Definition $3.1$ are easily verified. 
\end{proof}

\begin{remark}
Usually, the following two scenarios can occur for $e$ and $e'$:
\begin{itemize}
\item $e$ is a (scaled) binary edge from $b$ to $a$, and $e'$ is a (scaled) (modified) arithmetic edge with output $a$ and inputs: $ha + (b - ha) = b$, $b + (b-ha)$, $b + 2(b-ha)$, etc., for some $h \geq 1$. The arithmetic progression giving the nodes of $e'$ may be both finite or infinite. We note that $w(e) = w(e')$ since $\gcd(a, b) = \gcd(a, b-ha)$.

\item $e$ is a modified arithmetic edge with $k \in \N$ inputs, and $e'$ is the same edge but with more (possibly infinitely many) inputs.
\end{itemize}

We note that enrichment does not necessarily make a reduction graph better or more general, since it may be a more challenging problem to study a numerical semigroup generated by fewer numbers; in particular, an enrichment which adds the node $\langle 1 \rangle$ to a reduction graph $G$ turns $S(G)$ into $\langle 1 \rangle$, which is not very interesting. Unlike artificial nodes, enrichment should be used to find new graph-solvable semigroups rather than to study a pre-established semigroup.
\end{remark}

\begin{example}
The reduction graph in Figure $22$ contains two (scaled) binary edges, which can be enriched to obtain two (scaled) arithmetic edges. Indeed, given $p, q \in \N \cup \{\infty\}$, one has the total reduction graph:

\begin{figure}[h]
\centering
\includegraphics[scale = 0.8]{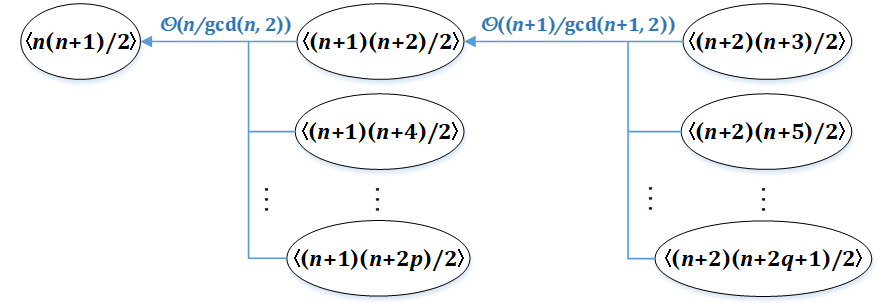}
\caption{Reduction Graph for Extended Triangular Numbers}
\end{figure}
\FloatBarrier

A similar enrichment can be applied to the case of tetrahedral numbers (in particular, Figure $23$, right). Example $5.6$ will be essential for proving Theorem $1.3$, in the next section. Other examples of enrichment will be presented within the proofs of Theorems $1.1$ and $1.2$.
\end{example}

\section{New Classes of Graph-Solvable Numerical Semigroups}

Here we prove Theorems $1.1$ to $1.5$, using the techniques developed in the previous sections:

\begin{myproof1}
We start with the case $d = b$, when the semigroups we wish to study become $S_1 = \left\langle a^n, a^n + a^{n-1}b, \ldots, a^n + ab^{n-1},  a^n + b^n \right\rangle$ (when $a \geq b$), and $S_2 = \left\langle a^n, a^n + a^{n-1}b, \ldots, a^n + a^{n-1}b + \ldots + b^n \right\rangle$. Observe that both semigroups are numerical since $\gcd(a, b) = 1$, and that for $1 \leq i \leq 2$, $S_i$ has the form $\left\langle a^n, a^n + A_{i,1}, a^n + A_{i,2}, \ldots, a^n + A_{i,n} \right\rangle$, where $A_{i,1} = a^{n-1}b$ and $A_{i,2}, \ldots, A_{i,n}$ are certain positive integers depending on $a, b, n$. It turns out that each $S_i$ is a free numerical semigroup, graph-solvable using one binary edge and $n-1$ linear-binary edges illustrated in the figure below (\emph{for now, disregard the interrupted lines}):

\begin{figure}[h]
\centering
\includegraphics[scale = 0.82]{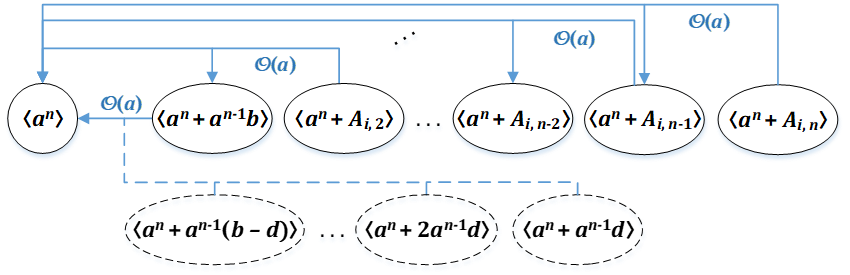}
\caption{Reduction Graph for Arithmetic-Geometric Sums}
\end{figure}
\FloatBarrier

Let us verify the linear-binary reduction edges used above, the $k^{th}$ of which has input $A_{i, k+1}$ and outputs $a^n$, $a^n + A_{i, k}$, for $1 \leq k \leq n-1$. According to the rules in Figure $20$, one can check that:
\begin{equation}
\frac{\gcd\left(a^n, a^n + A_{i, k}\right)}{\gcd\left(a^n, a^n + A_{i, k}, a^n + A_{i, k+1}\right)} = \frac{a^{n-k}}{a^{n-k-1}} = a,
\end{equation}
\begin{equation}
aA_{i, k+1} = M_i a^n + N_i A_{i, k} \in \langle a^n, A_{i, k}\rangle,
\end{equation}

for some nonnegative integers $M_i$ and $N_i$. Indeed, one can take $M_1 = a-b$ and $N_1 = b$ (recall that $a \geq b$ in this case), respectively $M_2 = a$ and $N_2 = b$. Hence the linear-binary reductions are valid. We note that the product of the weights used in Figure $25$ is $a^n$, which coincides with the root generator, hence the illustrated reduction graph has balance $1$. Since the two semigroups are numerical, Theorem $3.1$ applies.

The next step is to enrich the binary edge from $a^n + a^{n-1}b$ to $a^n$ in order to form a scaled arithmetic edge with common difference $a^{n-1}d$ for some $d \divides b$, illustrated with interrupted lines in Figure $25$. Note that $\gcd(a, d) = \gcd(a, b) = 1$, so the weight of the reduction is preserved and the reduction graph remains total. In this final version, the generators of the studied semigroup are precisely those listed in Theorem $1.1$, where $S_1$ and $S_2$ receive the additional generators: $a^n + a^{n-1}d, a^n + 2a^{n-1}d, \ldots, a^n + a^{n-1}(b-d)$.

It remains to compute the Frobenius numbers of $S_1$ and $S_2$, using Corollary $3.1$. Let $e_0$ denote the arithmetic edge with output $a^n$ and $\frac{b}{d}$ inputs of common difference $a^{n-1}d$, and let $e_{i,k}$ denote the linear-binary edge described by $(79)$ and $(80)$ for the semigroup $S_i$, where $1 \leq i \leq 2$ and $1 \leq k \leq n-1$. Then by relations $(37)$, $(61)$ and $(15)$ (accounting for scalings), we have:
\begin{equation}
\begin{split}
F(S_i) &= -a^n + \max(\Rem(e_0)) + \sum_{k = 1}^{n-1} \max(\Rem(e_{i, k})) \\
&= -a^n + a^{n-1}\left(a\lc \frac{(a-1)d}{b} \rc + d(a-1)\right) + \sum_{k = 1}^{n-1} A_{i, k+1}(a-1)
\end{split}
\end{equation}

This completes the proof of Theorem $1.1$, up to the computation of the sums $\sum_{k = 1}^{n-1} A_{i, k+1}$ for $1 \leq i \leq 2$, which we skip here since it is just a matter of summing geometric series.
\end{myproof1}

\begin{remark}
The case $d = b = 1$, $a = 2$ of Theorem $1.1$ yields a semigroup generated by the sequence $2^n, 2^n + 2^0, 2^n + 2^1, \ldots, 2^n + 2^{n-1}$, which leads to an interesting generalization:
\end{remark}

\begin{myproof2}
Let $0 \leq k \leq \nu_2(n)$. We will simultaneously find reduction graphs for the semigroups $S_k := \langle n, n+2^0, \ldots, n+2^k \rangle$ and $S_{k+1} := \langle n, n+2^0, \ldots, n+2^{k+1} \rangle$, such that the latter is an enrichment by an arithmetic edge of the former. The two graphs are illustrated below in a single figure, where the interrupted lines constitute the enrichment:

\begin{figure}[h]
\centering
\includegraphics[scale = 0.85]{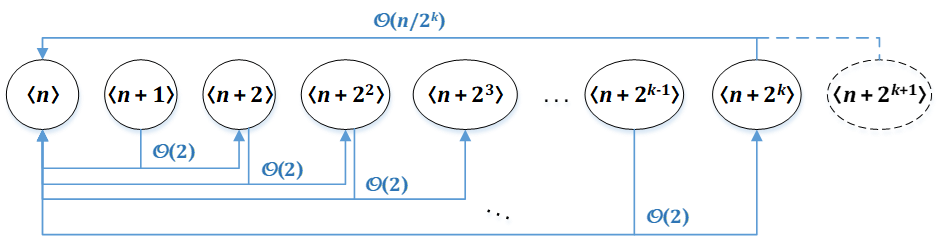}
\caption{Reduction Graph for Shifted Powers of $2$}
\end{figure}
\FloatBarrier

The figure uses $k$ linear-binary edges motivated by $2(n + 2^i) = n + (n + 2^{i+1})$ for $0 \leq i \leq k-1$, and one (scaled) binary reduction edge from $n + 2^k$ ro $n$, which can be enriched to obtain a modified arithmetic reduction. The balance of both reduction graphs is $\frac{2^k}{2^k} = 1$, and they both describe numerical semigroups ($S_k$, respectively $S_{k+1}$) since $\gcd(n, n+1) = 1$. Hence, we can apply relations $(37)$, $(61)$ and $(15)$ (accounting for scalings) to find the Frobenius numbers of $S_k$ and $S_{k+1}$:
\begin{align}
F(S_k) &= -n + (n+1) + (n+2) + \ldots + (n+2^{k-1}) + \left(n+2^k\right)\left(\frac{n}{2^k} - 1\right) \\
F(S_{k+1}) &= -n + (n+1) + (n+2) + \ldots + (n+2^{k-1}) + \left(n\lc \frac{\frac{n}{2^k} - 1}{2} \rc + n - 2^k\right)
\end{align}

Since $k \leq \nu_2(n)$, we know that $\frac{n}{2^k}$ is an integer, which is odd if and only if $k = \nu_2(n)$. Therefore, the ceiling involved in relation $(83)$ equals $\frac{n - 2^k}{2^{k+1}}$ if $k = \nu_2(n)$, and $\frac{n}{2^{k+1}}$ otherwise.

Now suppose $0 \leq k \leq \nu_2(n)$. One can check that relation $(82)$ produces the same result as relation $(83)$ after the substitution $k \to k-1$ (when $k \neq 0$), which is that:
\begin{equation}
F(S_k) = \frac{n^2}{2^k} + (k-1)n - 1
\end{equation}

In the remaining case when $k = \nu_2(n) + 1$, our only option is to use relation $(83)$ after the same substitution $k \to k-1$, in order to compute that:
\begin{equation}
F(S_k) = \frac{n^2}{2^k} + \left(k - \frac{3}{2}\right)n - 1
\end{equation}
\end{myproof2}

\begin{remarks}
One way to generalize Theorem $1.2$ is to scale each generator $n + 2^i$, for $0 \leq i \leq k$, by some odd positive integer $a_i$ such that $a_k \divides a_{k-1} \divides \ldots \divides a_0$ and $\gcd(a_0, n) = 1$. We mention that if $k = \nu_2(n) + 1$, one must also require that $a_k = a_{k - 1}$. Obtaining this more general result is a quick application of partial scaling; the structure of the reduction graph is identical to that from Figure $26$, and the computations are left to the reader.
\end{remarks}

\begin{myproof3}
Let $n \geq 1$, $k \geq 3$, and consider the semigroup $S := \left\langle \left\{ \frac{(n+i)(n+i\% 2+1)}{2} : 0 \leq i \leq k \right\}\right\rangle$, where $i\% 2 \in \{0, 1\}$ such that $i \equiv i\%2 \text{ (mod $2$)}$. It can be checked \cite{Triangular} that the first $4$ generators of this semigroup (given by $0 \leq i \leq 3$) have greatest common divisor $1$, hence $S$ is numerical. 

Moreover, $S$ is described by the total reduction graph from Figure $24$, when $p := \lf \frac{k}{2} \rf$ and $q := \lf \frac{k-1}{2} \rf$; we use these notations henceforth. It remains to apply Corollary $3.1$, in particular relations $(37)$ and $(61)$, to compute the Frobenius number of $S$ in two cases:

\textbf{\emph{Case} 1.} $n$ is even. Then the weights of the two arithmetic edges from Figure $24$ become $\frac{n}{2}$, respectively $n+1$, so by accounting for scaling we obtain the following:
\begin{equation}
\begin{split}
F(S) &= - \frac{n(n+1)}{2} + (n+1)\left( \frac{n}{2} \lc \frac{\frac{n}{2} - 1}{p} \rc + \frac{n}{2} - 1 \right) + \frac{n+2}{2}\left( (n+1) \lc \frac{n+1 - 1}{q} \rc + 2(n+1 - 1) \right) \\
&= \lc \frac{n - 2}{2p} \rc T_n + \lc \frac{n}{q} \rc T_{n+1} + n^2 + n - 1
\end{split}
\end{equation}

Recall that we are using the notation $T_n = \frac{n(n+1)}{2}$ for the $n^{th}$ triangular number.

\textbf{\emph{Case} 2.} $n$ is odd. Then the weights of the arithmetic edges from Figure $24$ become $n$ and respectively $\frac{n+1}{2}$. Using relations $(37)$ and $(61)$ adjusted for scaling, we get that:
\begin{equation}
\begin{split}
F(S) &= - \frac{n(n+1)}{2} + \frac{n+1}{2}\left( n \lc \frac{n - 1}{p} \rc + 2(n - 1) \right) + (n+2)\left( \frac{n+1}{2} \lc \frac{\frac{n+1}{2} - 1}{q} \rc + \frac{n+1}{2} - 1 \right) \\
&= \lc \frac{n - 1}{p} \rc T_n + \lc \frac{n-1}{2q} \rc T_{n+1} + n^2 - 2
\end{split}
\end{equation}
\end{myproof3}

\begin{remark}
Since the reduction graph used for Theorem $1.3$ only contains two edges, one can use our MAPLE script \cite{Program} to compute the Frobenius number of the studied semigroup. The four cases given by the possible parities of $n$ and $k$ are listed as comments in the "LIST OF REDUCTIONS USED" section of the program, marked by the phrase "Extended Triangular".
\end{remark}

Moving on, the proofs of our last two theorems illustrate a common idea: the composition by the root of several edges of the same type is a great tool to study semigroups whose generators are related to the prime factorization of a fixed integer (such as those given by multiplicative functions):

\begin{myproof4}
In this proof we will cheat a little by using a reduction edge that we have not proved in this paper, although we mentioned it in the table from the end of Section $4$. In \cite{Repunit}, the authors find the exact form of an Ap\'ery set $A$ corresponding to a semigroup generated by \emph{repunit} numbers, i.e.:
\begin{equation}
\left\langle \frac{b^n - 1}{b - 1}, \frac{b^{n+1} - 1}{b - 1}, \frac{b^{n+2} - 1}{b - 1}, \ldots \right\rangle = \left\langle \frac{b^n - 1}{b - 1} \right\rangle \oplus A
\end{equation}

For now, all we need to know is that $\max(A) = \frac{b^{2n}-1}{b-1}-1$. Equation $(88)$ can be seen as an Ap\'ery reduction of weight $|A| = \frac{b^n - 1}{b-1}$, as shown in Figure $27$ (left):

\begin{figure}[h]
\centering
\includegraphics[scale = 0.82]{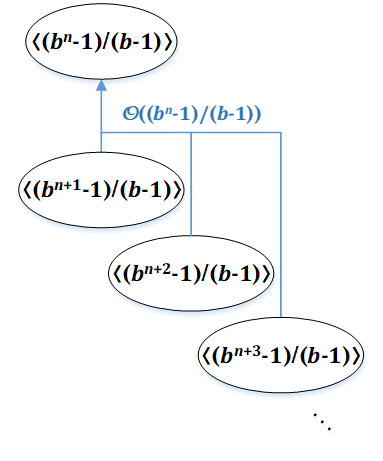}
\hspace{0.5cm}
\includegraphics[scale = 0.82]{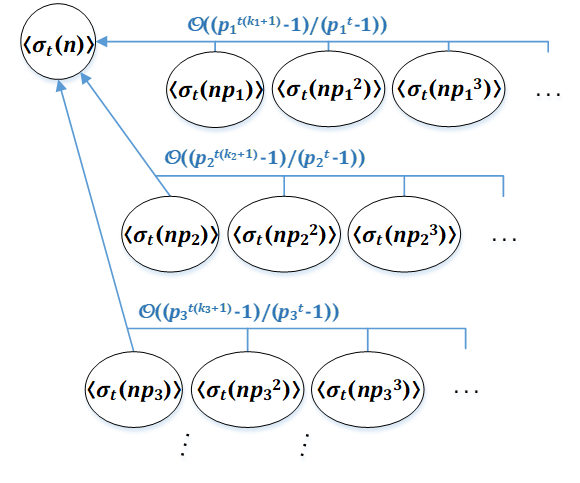}
\caption{Repunit Reduction Edge (left) vs. Reduction Graph for Divisor Functions (right)}
\end{figure}
\FloatBarrier

Now fix a positive integer $n$ and consider a repunit reduction edge $e_i$ for each maximal prime power $p_i^{k_i} > 1$ dividing $n$, such that the output of $e_i$ is $\sigma_t\left(p_i^{k_i}\right) = \frac{p_i^{t(k_i+1)}-1}{p_i^t-1}$, and the inputs of $e_i$ are given by $\sigma_t\left(p_i^k\right) = \frac{p_i^{t(k+1)+1}}{p_i^t-1}$ for $k > k_i$. By considering the composition by the root $e_1 \circ e_2 \circ \ldots \circ e_l$ where $p_1, \ldots, p_l$ are all the prime factors of $n$, one obtains the reduction graph illustrated in Figure $27$ (right), in light of the following formula specific to multiplicative functions:
\begin{equation}
\sigma_t(n) = \prod_{1 < p^k \divides \divides n} \sigma_t(p^k) = \prod_{1 < p^k \divides \divides n} \frac{p^{t(k+1)} - 1}{p^t - 1}
\end{equation}

According to Lemma $5.2$, one has $\Bal(e_1 \circ e_2 \circ \ldots \circ e_l) = \Bal(e_1) \cdot \ldots \cdot \Bal(e_l) = 1$. The resulting reduction graph spans the semigroup $S := \left\langle \left\{ \sigma_t(m) : \frac{m}{n} \text{ is a prime power}\right\} \right\rangle$, where we consider $1$ to be a prime power ($1 = p^0$); note that if $\frac{m}{n}$ is a power of a prime that does not divide $n$, one has $\sigma_t(n) \divides \sigma_t(m)$, so $\sigma_t(m) \in S$ anyway. Therefore, as long as $S$ is numerical (which follows from the coprimality restrictions in the statement of Theorem $1.4$), one can apply Corollary $3.1$ and the results about repunit Ap\'ery sets \cite{Repunit} to conclude that:
\begin{equation}
\begin{split}
F(S) &= -\sigma_t(n) + \sum_{1 < p^k \divides \divides n} \frac{\sigma_t(n)}{\frac{p^{t(k+1)}-1}{p^t-1}} \left( \frac{p^{2t(k+1)}-1}{p^t - 1} - 1 \right) \\
&= \sigma_t(n) \left( -1 + \sum_{1 < p^k \divides \divides n} \frac{p^{2t(k+1)}-p^t}{p^{t(k+1)}-1} \right)
\end{split}
\end{equation}

\end{myproof4}

\begin{remark}
To ensure that $S$ is a numerical semigroup, Theorem $1.4$ requires that for any distinct $p^k, q^l \divides \divides n$, one has $\gcd\left(\sigma_t(p^k), \sigma_t(q^l)\right) = 1$; denote this property of a positive integer $n$ by $P(n)$.

One may wonder if $P(n)$ is true often enough; luckily, a short argument which we will not bother presenting here shows that for any positive integer $n$ such that $P(n)$ is true, one can find a prime power $p^k$ such that $p \not\divides n$ and $P(np^k)$ is true. Since $P(n)$ is vacuously true whenever $n$ is a prime power, this generates a rich infinity of positive integers for which our reasoning applies.
\end{remark}

\begin{myproof5}
Let $n$ be a positive integer with prime factorization $n = p_1^{k_1} p_2^{k_2} \ldots p_l^{k_l}$. Then consider the following reduction graphs of balance $1$:

\begin{figure}[h]
\centering
\includegraphics[scale = 0.82]{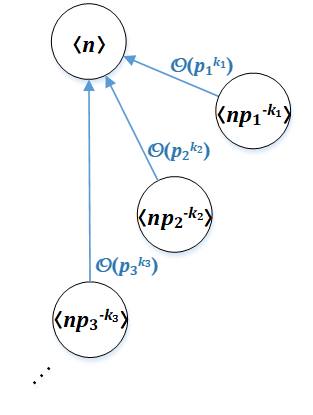}
\hspace{1cm}
\includegraphics[scale = 0.82]{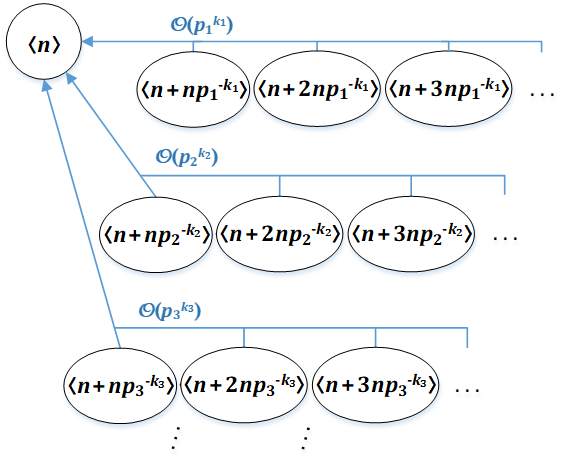}
\caption{Reduction Graphs for Almost Divisible Numbers}
\end{figure}
\FloatBarrier

One can see the graph on the left as the composition by the root of $l$ binary edges of the form $\langle p_i^{k_i} \rangle \gets \langle 1 \rangle$ (which reminds of Example $5.4$), and the graph on the right as the composition by the root of $l$ infinite arithmetic edges of the form $\langle p_i^{k_i} \rangle \gets \langle p_i^{k_i} + 1 \rangle, \langle p_i^{k_i} + 2 \rangle, \langle p_i^{k_i} + 3 \rangle, \ldots$, with $1 \leq i \leq l$.

The first graph describes the semigroup $\left\langle \left\{ \frac{n}{p^k} : p^k \divides\divides n \right\} \right\rangle = \langle \{ m \in \N : m \leq n, m::n\} \rangle =: S_\leq$, which is the same as $\langle\{ m \in \N : m::n \}\rangle $. Recall from Section $1$ that we say $m :: n$ when $m$ is \emph{almost divisible} by $n$, i.e. the denominator of $\frac{m}{n}$ in reduced terms is a prime power. It is easy to see that $S_\leq$ is numerical, since $\gcd\left(np_1^{-k_1}, \ldots, np_l^{-k_l}\right) = 1$. Therefore, Corollary $3.1$ together with relation $(15)$ (after scaling) imply that:
\begin{equation}
\begin{split}
F(S_\leq) &= -n + \sum_{1 < p^k \divides\divides n} \frac{n}{p^k} \left(p^k - 1\right) \\
&= n \left( -1 + \sum_{1 < p^k \divides \divides n} \frac{p^k - 1}{p^k}\right)
\end{split}
\end{equation}

Similarly, the reduction graph on the right of Figure $28$ describes the second semigroup $S_\geq = \langle \{m : m \geq n, m::n \} \rangle$, since $m::n$ if and only if $(m-n)::n$, so any such $m$ can be written as $m = n + u \cdot np^{-k}$ for some $u \geq 0$ and $p^k \divides\divides n$. Again, it is easy to see that $S_\geq$ is numerical, since $\gcd\left(n, n + np_1^{-k_1}, \ldots, n + np_l^{-k_l}\right) = \gcd\left(n, np_1^{-k_1}, \ldots, np_l^{-k_l}\right) = 1$. Hence by Corollary $3.1$ and $(67)$ (adjusted for scaling), we obtain that: 
\begin{equation}
\begin{split}
F(S_\leq) &= -n + \sum_{1 < p^k \divides\divides n} \frac{n}{p^k} \left(p^k + p^k - 1\right) \\
&= n \left( -1 + \sum_{1 < p^k \divides \divides n} \frac{2p^k - 1}{p^k}\right),
\end{split}
\end{equation}
which completes our proof.
\end{myproof5}

\section{Final Remarks}

\begin{enumerate}

\item Perhaps the biggest advantage of the method presented in this paper is that it is constructive: any result of a paper studying the Ap\'ery set of a certain numerical semigroup can be translated into a basic reduction, and then integrated into larger reduction graphs to study even more complicated semigroups (as in the proof of Theorem $1.4$, where we went from repunit numbers to divisor functions).

\item In the previous section, we focused on finding the Frobenius numbers of a few numerical semigroups. With more work (using Theorem $3.1$ and the rest of Corollary $3.1$), one can also compute the Ap\'ery set, Hilbert series and genus in each case (or even the gaps' power sums). In particular, in order to compute the genus, it might be easier to compute the asymmetry first, since most of the edges used are symmetric.

\item If for various reasons the reader wants to study the zeros of the generating function of an Ap\'ery set, this can be done easily with reduction graphs: since the Ap\'ery set is the direct sum of the remainder sets, the zeros of the Ap\'ery set's generating function are the union with multiplicity of the zeros of the remainder sets' generating functions.

\item The general question "\emph{Which numerical semigroups are graph-solvable?}" is subjective, since it depends on what we consider to be an acceptable basic reduction. A better question would be "\emph{Which numerical semigroups are graph-solvable using certain types of edges?}". For instance, if we limit ourselves to linear and (scaled) binary basic reductions, the answer is precisely the free numerical semigroups. More generally, if we only use symmetric edges, we can only describe symmetric numerical semigroups. On the other extreme, any numerical semigroup can be associated with one Ap\'ery reduction edge, a perspective that is useful for proving recursive formulae for general numerical semigroups (recall Corollary $5.1$ about the formula of Brauer and Shockley \cite{Frobenius}).

Another interesting question could be "\emph{Which numerical semigroups have a total reduction graph with more than one nontrivial edge?}", so that we eliminate the case of single Ap\'ery reductions. As suggested in the beginning of Section $4$, if the chosen root generator $r(G)$ can be expressed as a nontrivial product of other integer variables (e.g., $a^n = a \cdot a \cdot \ldots \cdot a$), there is a good chance that the semigroup has a total multi-edge reduction graph. There are some cases, however, when one can build such a reduction graph without decomposing $r(G)$ (recall the Fibonacci triplets), by using only one edge of weight $r(G)$ and linear edges of weight $1$ in the rest. This may require, as in Proposition $4.1$, to use at least one node that is not a monogenic semigroup.

Otherwise, if building a multi-edge reduction graph for a given semigroup seems impossible, the method of linear exchanges that we developed in Section $4$ can serve as complementary to the graphical method, in order to find basic reductions rather than clever networks of edges.

\end{enumerate}

\textbf{Acknowledgements.} The author is deeply grateful to Professor Terence Tao for his helpful insights and suggestions during the development of this article.

\end{document}